\documentclass[11pt,amsfonts]{article}
\usepackage[margin=1in]{geometry}  
\usepackage{graphicx}              

\usepackage{amsmath, amssymb, amsthm, epsfig}               
\usepackage{enumerate}
\usepackage{amsfonts}              
\usepackage{amsthm}                
\usepackage{amsthm}
\usepackage{enumerate}
\theoremstyle{plain}
\usepackage{color}
\newtheorem{corollary}{Corollary}[section]

\newtheorem{lemma}[corollary]{Lemma}
\newtheorem{prop}[corollary]{Proposition}
\newtheorem{rem}[corollary]{Remark}
\newtheorem{thm}[corollary]{Theorem}

\newfont{\sBlackboard}{msbm10 scaled 900}

\newcommand{\mylabel}[1]{\label{#1}
    \ifx\undefined\stillediting
    \else \fbox{$#1$}\fi }
\newcommand{\BE}{\begin{equation}}

\newcommand{\EEQ}{\end{equation}}
\newcommand{\rfb}[1]{\mbox{\rm
        (\ref{#1})}\ifx\undefined\stillediting\else:\fbox{$#1$}\fi}

\newfont{\Blackboard}{msbm10 scaled 1200}

\newfont{\roma}{cmr10 scaled 1200}

\newcommand{\bb}{\begin{equation}}
\newcommand{\bbb}{\end{equation}}

\newcommand{\mm}    {{\hbox{\hskip 0.5pt}}}

\newcommand{\bluff} {{\hbox{\raise 15pt \hbox{\mm}}}}
scaled\magstep2

\makeatletter
\def\section{\@startsection {section}{1}{\z@}{-3.5ex plus -1ex minus
        -.2ex}{2.3ex plus .2ex}{\large\bf}}
\makeatother

\begin{document}
\title{Least-energy nodal solutions of nonlinear equations with fractional Orlicz-Sobolev spaces }
\author{Anouar Bahrouni, Hlel Missaoui and Hichem Ounaies}


\maketitle
\begin{abstract}
In this paper, we study the existence of least-energy nodal
(sign-changing)
 weak solutions for a class of fractional Orlicz equations given by

$$(-\triangle_{g})^{\alpha}u+g(u)=K(x)f(u),\ \ \text{in}\  \mathbb{R}^{N},$$
where $N\geq 3,\ (-\triangle_{g})^{\alpha}$ is the  fractional Orlicz $g$-Laplace
operator, while $f$ $\in C^{1}(\mathbb{R})$ and $K$ is a positive
and continuous function. Under a suitable  conditions on $f$ and $K$, we  prove a compact embeddings result for weighted fractional Orlicz-Sobolev spaces. Next, by a minimization argument on Nehari manifold and a quantitative deformation lemma, we show the existence of at least one nodal (sign-changing) weak solution.
\end{abstract}



\section{Introduction}
In this paper, we are concerned with the existence of leat-energy
nodal (sign-changing) weak solutions
 for the following  problem

\begin{equation}\label{eqn:einstein}
(-\triangle_{g})^{\alpha}u+g(u)=K(x)f(u),\ \ \text{in}\
\mathbb{R}^{N},\tag{P}
 \end{equation}
 where  $(-\triangle_{g})^{\alpha}$ stands the fractional $g-$Laplace operator, $N\geq 3$, $\alpha\in (0,1)$ and $g: \mathbb{R}\rightarrow \mathbb{R}$ will be specified later,
 while $f$ and $K$ are generally two continuous functions whose assumptions will be introduced in what follows.\\

 The novelty of our work (see Theorem \ref{thm1}) is the fact that we combine several different phenomena in one problem.  The
features of this paper are the following:
\begin{enumerate}
\item[$(1)$] the existence of nodal solutions.
\item[$(2)$] the presence of the new fractional $g-$Laplacian.
\item[$(3)$] the lack of compactness due to the free action of the translation group on $\mathbb{R}^{N}$.
\end{enumerate}
To the best of our knowledge, this is the first paper proving the
existence of sign-changing solutions with the combined effects
generated by the above features.\\

 Recently, non-local problems and operators have been widely studied in the literature
and have attracted the attention of lot of mathematicians coming
from different research areas. This type of non-local
 operators arises in the description of various
phenomena in the applied sciences, such as optimization, finance,
phase transitions, stratified materials, anomalous diffusion,
crystal dislocation, ultra-relativistic limits of quantum
mechanics... (see, \cite{21}) and the references therein. More
recently, the papers of Caffarelli et al.
  \cite{13,14,15}, led to a series  of works involving the fractional diffusion, $(-\triangle)^{\alpha},\ 0<\alpha<1$ (see, for instance,
  \cite{16,19,20,17,21,18,23,24}).
 In the last references, the authors tried to see which results
survive when the Laplacian is replaced by the fractional Laplacian.\\

A natural question is to see what results can be recovered when the
standard $p(x)-$Laplace operator is replaced by the fractional
$p(x)-$Laplacian.
 To our best knowledge, Kaufmann et al. \cite{34} firstly introduced some results on fractional Sobolev
spaces with variable exponent $W^{\alpha,q(.),p(.,.)}(\Omega)$ and
the fractional $p(x)-$Laplacian. There, the authors established
compact embedding theorems of these spaces
 into variable exponent Lebesgue spaces.
  In \cite{39}, Bahrouni and R\u{a}dulescu  obtained some further qualitative properties of the fractional Sobolev space $W^{\alpha,q(.),p(.,.)}(\Omega)$ and the fractional
  $p(x)-$Laplacian. After that, some studies on this kind of
problems have been performed by using different approaches, see \cite{40,48,41,49}.\\

The study of nonlinear elliptic equations involving quasilinear
homogeneous type operators is based on the theory of Sobolev spaces
in order to find weak solutions. In the case of nonhomogeneous
differential operators, the natural setting for this approach is the
use of Orlicz-Sobolev spaces. For more details on the theory of
Orlicz and Orlicz-Sobolev spaces, we can cite \cite{42,44,22} and
the references therein.  It is, therefore, a natural question to see
what results can be recovered when the g-Laplace operator is
replaced by the fractional g-Laplacian.\\

It is worth mentioning that there are some papers concerning related
equations involving the fractional $g-$Laplace operator and
fractional Orlicz-Sobolev spaces. In fact, results on this subject
are few. In $2019$, J. Fern\'andez. Bonder firstly introduced some
results on this context (see \cite{9}). In particular, the authors
generalize the $g-$Laplace operator to the fractional case. They
also introduce a suitable functional space to study an equation in
which a
 fractional $g-$Laplace operator is present. After that in 2020 some authors, like S. Bahrouni and A. M. Salort have continued these studies see \cite{4,29,11,28}.
 More precisely, they proved some basic results like embedding problems and some other fundamental properties.\\
 As we mentioned above, the literature survey on
problems involving the fractional $g-$Laplacian is
 almost meagre since it is still a work in progress.  To the best of
 our knowledge, there is no paper devoted to the studies of the existence of sign-changing solution for the equations contains the fractional $g-$Laplacian. \\

When $\alpha=1$ and $g(t)=|t|^{p-2}t$, Eq. \eqref{eqn:einstein} gives
back the classical quasilinear equation
$$-\Delta_{p}
u=f(x,u), \ \ \text{ where}\ \ \Delta_{p} u=\text{div} (|\nabla u|^{p-2}u).$$\\
We do not intend to review the huge bibliography of nodal solutions
for the above equations, we just emphasize that the nodal Nehari
 manifold approach is a very efficient method to prove the existence of
sign-changing solutions, see \cite{37,5,36,35,6,38,30}.\\ In \cite{5}, S. Barile and G. M. Figueredo
  have studied the following equation:
\begin{equation}\label{1.1}
-\text{div}(a(\vert \nabla u\vert^{p})\vert \nabla u\vert^{p-2}\nabla u)+V(x)b(\vert u\vert^{p})\vert u\vert^{p-2}=K(x)f(u),\ \ \text{in}\  \mathbb{R}^{N},\tag{$P_1$}
 \end{equation}
where $N\geq 3,\ 2\leq p<N,\ a,\ b,\ f$ are $C^{1}$ real functions
and $V,\ K$ are continuous positives functions. Under some
further assumptions on $f,\ a, \ b, \ K$ and $V$, and using the
nodal Nehari manifold method, they proved the existence of a
sign-changing solution with two nodal domains.
In \cite{6}, G. M. Figueredo considered   the following  equation:
   \begin{equation}\label{eqf}
-M\left( \int_{\Omega}g(\vert \nabla u\vert)dx\right)\Delta_g u=f(u),\ \text{in}\ \Omega \tag{$P_2$}
 \end{equation}
where $\Omega$ is a bounded domain in $\mathbb{R}^{N}$, $M$ is  $C^1$ function, $f$ is a superlinear $C^1$ class function with subcritical growth and $\Delta_g$ is the Orlicz $g$-Laplace operator. By using the nodal Nehari manifold method, he proved the existence of a nodal solution.\\

 In \cite{2,3}, V. Ambrosio et al. have  extended the result
obtained in \cite{5} to a class of  fractional problems (see also \cite{31,7} ).\\

To the best of our knowledge, there are no results concerning the
existence of sign-changing solutions for the equations in which the
fractional $g-$Laplace operator is present. Hence, a natural
question is whether or not there exist nodal solutions of problem
\eqref{eqn:einstein}. The goal of the present paper is to develop a
thorough qualitative analysis in this direction. Our method is
inspired by the work of
S. Barile and G. M. Figueredo \cite{5}.\\ 



The paper is organized as follows. In  Sect. 2, we give some
definitions and fundamental properties of Orlicz-Sobolev spaces and
fractional Orlicz-Sobolev spaces. In Sect. 3 we show a compact
embedding type result. The last section is divided into two parts:
In the first subsection, we prove some technical lemmas which will
be useful in the sequel. In the second part,  we conclude the proof
of our main result (see Theorem \ref{thm1}).
\section{Mathematical background and hypotheses}
In this section, we recall some necessary properties about Orlicz-Sobolev and fractional Orlicz-Sobolev spaces. For more details we refer the readers to \cite{44,22,28}.\\
We start by recalling some definitions about the well-known $N$-functions.
Let $g$ be a real-valued function defined on $\mathbb{R}$ and having the following properties:
\begin{enumerate}
\item[$(g_1)$]
\begin{enumerate}
\item $g(0)=0,\ g(t)> 0$ if $t>0$ and $\lim\limits_{t\rightarrow +\infty}g(t)=+\infty.$
\item $g$ is nondecreasing and odd function.
\item $g$ is right continuous.
\end{enumerate}
\end{enumerate}
The real-valued function $G$ defined on $\mathbb{R}$ by
$$G(t)=\int_{0}^{ t} g(s)\ ds$$
is called an $N$-function. It is easy to see that $G$ is even,
positive, continuous and  convex function. Moreover, one has
 $G(0)=0$,\\ $\frac{G(t)}{t}\rightarrow 0$ as $t\rightarrow0$ and  $\frac{G(t)}{t}\rightarrow +\infty$ as $t\rightarrow+\infty$.\\
The conjugate $N$-function of G, is defined by
$$\tilde{G}(t)=\int_{0}^{ t} \tilde{g}(s)\ ds,$$
where  $\tilde{g}: \mathbb{R}\rightarrow\mathbb{R}$ is given by
$\tilde{g}(t)=\sup\{s:\ g(s)\leq t\}$.\\
We have
\begin{equation}\label{9}
st\leq G(s)+\tilde{G}(t),
\end{equation}
which is known as the Young inequality. Equality in \eqref{9} holds  if and only if either $t=g(s)$ or $s=\tilde{g}(t)$.\\
We say that $G$  satisfies the $\triangle_2$-condition, if there exists $C>0$, such that
\begin{equation}
G(2t) \leq CG(t),\ \text{for all}\  t > 0.
\end{equation}
If $A$ and $B$ are two $N$-functions, we say that $A$ is essentially
stronger than $B$  $(B\prec\prec A$ in symbols$)$, if and only if for every positive constant $\lambda$, we have
\begin{equation}
\lim_{t\rightarrow+\infty}\frac{B(t)}{A(\lambda t)}=0.
\end{equation}
 In  what follows, we assume that the
$N$-functions $\tilde{G}$ satisfies the $\triangle_2$-condition and
we suppose that  $G$ and $g$ are of class $C^{1}$ and satisfy the
two following conditions
\begin{enumerate}
\item[$(g_2)$] There exist $L$ and $m$ such that $L\in(\frac{N}{2},N)$, and $L<m<\min\lbrace\frac{L^{*}}{2},N\rbrace$ where $L^{*}=\frac{LN}{N- L}$ such that
$$L\leq \frac{g(t)t}{G(t)}\leq m,\ \text{for all}\ t>0.$$
\item[$(g_3)$] $\displaystyle \int_{0}^{1}\frac{G^{-1}(t)}{t^{\frac{N+\alpha}{\alpha}}}dt<\infty$, $\displaystyle
\int_{1}^{+\infty}\frac{G^{-1}(t)}{t^{\frac{N+\alpha}{\alpha}}}dt=\infty$
and
 $$L-1\leq \frac{g^{'}(t)t}{g(t)}\leq m-1,\ \text{for all}\ t>0.$$
\end{enumerate}
The assumption $(g_2)$  implies that $G$  satisfies the $\triangle_2$-condition.\\

The Orlicz space $L^{G}(\mathbb{R}^{N})$  is
  the vectorial space of  measurable functions $u:\mathbb{R}^{N}\rightarrow\mathbb{R}$ such that
$$\rho(u)=\int_{\mathbb{R}^{N}}G(u)\ dx < \infty,\ \ (\text{see}\ \cite{1,8}).$$
  \\
$L^{G}(\mathbb{R}^{N})$ is a Banach space under the Luxemburg norm
$$\Vert u\Vert_{(G)}=\inf \left\lbrace  \lambda >0\ : \ \rho(\frac{u}{\lambda}) \leq 1\right\rbrace. $$\\
Next, we introduce the Orlicz-Sobolev spaces and the fractional Orlicz-Sobolev spaces.\\
 We denote by
$W^{1,G}(\mathbb{R}^{N})$ the Orlicz-Sobolev space defined by
$$W^{1,G}(\mathbb{R}^{N}):=\bigg{\{}u\in L^{G}(\mathbb{R}^{N}):\ \frac{\partial u}{\partial x_{i}}\in L^{G}(\mathbb{R}^{N}),\ i=1,...,N\bigg{\}}.$$
$W^{1,G}(\mathbb{R}^{N})$ is a Banach space with respect to the norm
$$\|u\|_{1,G}=\|u\|_{(G)}+ \||\nabla u|\|_{(G)}.$$
We denote by $W^{\alpha,G}(\mathbb{R}^{N})$ the fractional Orlicz-Sobolev space defined by
\begin{equation}\label{20}
    W^{\alpha,G}(\mathbb{R}^{N})=\bigg{\{}u\in
L^{G}(\mathbb{R}^{N}):\ \overline{\rho}(u)<\infty\bigg{\}},
    \end{equation}
    where $0<\alpha<1$ and $\displaystyle{\overline{\rho}(u)=\int_{\mathbb{R}^{N}}\int_{\mathbb{R}^{N}}G\left( \frac{u(x)-u(y)}{\vert x-y\vert^{\alpha}}\right)\frac{dxdy}{\vert x-y\vert^{N}}}.$\\
    We provide  $W^{\alpha,G}(\mathbb{R}^{N})$  with the norm
    \begin{equation}\label{21}
    \|u\|_{\alpha,G}=\|u\|_{(G)}+[u]_{(\alpha,G)},
    \end{equation}
    where $[.]_{(\alpha,G)}$ is the Gagliardo semi-norm, defined by
    \begin{equation}\label{22}
    [u]_{(\alpha,G)}=\inf\bigg{\{}\lambda>0:\ \overline{\rho}(\frac{u}{\lambda})\leq1\bigg{\}}.
    \end{equation}
Since $G$ and $\tilde{G}$ satisfy the $\triangle_{2}$-condition, then the fractional Orlicz-Sobolev spaces   is a reflexive separable Banach space.\\
We  define another norm on $W^{\alpha,G}(\mathbb{R}^{N})$
\begin{equation}
\Vert u\Vert=\inf\left\lbrace \lambda>0 : \tilde{\rho}(\frac{u}{\lambda})\leq 1\right\rbrace
\end{equation}
where $\tilde{\rho}(u)=\rho(u)+\overline{\rho}(u). $ The two norms
$\Vert u\Vert$ and $\|u\|_{\alpha,G}$ are equivalent. More
precisely, we have
 $$\frac{1}{2}\|u\|_{\alpha,G}\leq \Vert u\Vert\leq 2\|u\|_{\alpha,G},\ \text{for all}\ u\in W^{\alpha,G}(\mathbb{R}^{N}),$$
 (see \cite{29} Lemma 3.5 page 11).\\

Another $N$-function related to function $G,$ is the Sobolev conjugate function $G_{*}$  defined by 
\begin{equation}\label{gstar}G_{*}^{-1}(t)=\int_{0}^{t}\frac{G^{-1}(s)}{s^{\frac{N+\alpha}{N}}}\
ds, \ t>0.
\end{equation}
Now, we are ready to recall  a variant of continuous and compact
embedding theorem of $W^{\alpha,G}(\mathbb{R}^{N})$  in Orlicz
spaces.
 \begin{thm}\label{thm2} $($see \cite{29}$)$
 \begin{enumerate}
 \item If $\Omega$ is an open bounded set in $\mathbb{R}^{N},$ then the embedding
 $$W^{\alpha,G}(\Omega)\hookrightarrow L^{G}\Omega)$$
 is compact.
 \item
If $B$ is an $N$-function satisfies the $\triangle_{2}$-condition
such that $B \prec\prec G_{*}$ ($G_{*}$ is given in \eqref{gstar}),
then the embedding
$$W^{\alpha,G}(\mathbb{R}^{N})\hookrightarrow L^{B}(\mathbb{R}^{N})$$ is continuous.
In particular, we have the continuous embeddings $$W^{\alpha,G}(\mathbb{R}^{N})\hookrightarrow L^{G}(\mathbb{R}^{N})$$
and
$$W^{\alpha,G}(\mathbb{R}^{N})\hookrightarrow L^{G_{*}}(\mathbb{R}^{N}).$$

 \end{enumerate}
\end{thm}
\ \\

The fractional Orlicz $g$-Laplacian operator is defined as
\begin{equation}\label{17}
(-\triangle)^{\alpha}_{g}u(x)= P.V.\int_{\mathbb{R}^{N}} g\bigg{(}\frac{u(x)-u(y)}{|x-y|^{\alpha}}\bigg{)}\frac{dy}{|x-y|^{N+\alpha}},
\end{equation}
where $P.V.$ is the principal value.\\
The operator $(-\triangle)^{\alpha}_{g}$ is well defined between $W^{\alpha,G}(\mathbb{R}^{N})$ and its dual space  $W^{-\alpha,\tilde{G}}(\mathbb{R}^{N})$. In fact, we have that
\begin{equation}\label{18}
\langle(-\triangle)^{\alpha}_{g}u,v\rangle=\int_{\mathbb{R}^{N}}\int_{\mathbb{R}^{N}} g\bigg{(}\frac{u(x)-u(y)}{|x-y|^{\alpha}}\bigg{)}\frac{v(x)-v(y)}{|x-y|^{\alpha}}\frac{dxdy}{|x-y|^{N}},
\end{equation}
for all $u, v\in W^{\alpha,G}(\mathbb{R}^{N})$ see [\cite{9}, Theorem 6.12].\\

Under the assumptions  $(g_1)-(g_3)$, some elementary inequalities and properties listed in the following lemmas are valid.

\begin{lemma}\label{lem444}[see \cite{9} Lemma 2.4]\\
If $G$ is $N$-function, then
$$G(a+b)\geq G(a)+G(b),\ \text{for all}\ a,b\geq 0.$$
\end{lemma}

\begin{lemma}\label{lem2.3}[see \cite{11}, Lemma 3.1 and \cite{10} Lemma 2.1 ]\\
Assume $(g_1)$ and $(g_2)$ hold,  then
$$ \min\lbrace a^{L},a^{m}\rbrace G(t)\leq G(at)\leq \max\lbrace a^{L},a^{m}\rbrace G(t)\ \text{for}\ a,t\geq 0,$$

$$ \min\lbrace \Vert u\Vert_{(G)}^{L},\Vert u\Vert_{(G)}^{m}\rbrace \leq \rho(u)\leq \max\lbrace \Vert u\Vert_{(G)}^{L},\Vert u\Vert_{(G)}^{m}\rbrace\ \text{for}\ u \in L^{G}(\mathbb{R}^{N})$$
and
$$ \min\lbrace [u]_{\alpha,G}^{L},[u]_{\alpha,G}^{m}\rbrace \leq \overline{\rho}(u)\leq \max\lbrace [u]_{\alpha,G}^{L},[u]_{\alpha,G}^{m}\rbrace, \ \ \text{for}\ u \in W^{\alpha,G}(\mathbb{R}^{N}). $$

\end{lemma}
\begin{lemma}\label{lem2.4}[see \cite{10} Lemma 2.1]\\
Assume $(g_3)$ holds, then
$$ \min\lbrace a^{L-1},a^{m-1}\rbrace g(t)\leq g(at)\leq \max\lbrace a^{L-1},a^{m-1}\rbrace g(t)\ \text{for}\ a,t\geq 0.$$
\end{lemma}
\begin{lemma}\label{lem2.5}[see \cite{10} Lemma 2.4]\\
The function $G_*$ satisfies the following inequality
\begin{equation}\label{2.7}
L^{*}\leq \frac{g_*(t)t}{G_*(t)}\leq m^{*}\ \text{for} \ t>0,
\end{equation}
where $\displaystyle{G_{*}(t)=\int_{0}^{t}g_{*}(s)\ ds},$  $m^{*}=\frac{mN}{N- m}$ and $L^{*}=\frac{LN}{N- L}.$
\end{lemma}
As an immediate consequence of  Lemma \ref{lem2.5}, we get the
following result:
\begin{lemma}\label{lem2.6}[see \cite{10} Lemma 2.2]\\
Assume $(g_1)-(g_2)$ hold, then
$$ \min\lbrace a^{L^{*}},a^{m^{*}}\rbrace G_*(t)\leq G_*(at)\leq \max\lbrace a^{L^{*}},a^{m^{*}}\rbrace G_*(t)\ \text{for}\ a,t\geq 0,$$

$$ \min\lbrace \Vert u\Vert_{(G_*)}^{L^{*}},\Vert u\Vert_{(G_*)}^{m^{*}}\rbrace \leq \int_{\mathbb{R}^{N}}G_*(u)\ dx \leq \max\lbrace \Vert u\Vert_{(G_*)}^{L^{*}},
\Vert u\Vert_{(G_*)}^{m^{*}}\rbrace\ ,\text{for}\ u \in
L^{G_*}(\mathbb{R}^{N})$$ and
$$ \min\lbrace [u]_{\alpha,G_*}^{L^{*}},[u]_{\alpha,G_*}^{m^{*}}\rbrace \leq \int_{\mathbb{R}^{N}}\int_{\mathbb{R}^{N}}G_*\left( \frac{u(x)-u(y)}{\vert x-y\vert^{\alpha}}\right) \frac{dxdy}{\vert x-y\vert^{N}} \leq \max\lbrace [u]_{\alpha,G_*}^{L^{*}},[u]_{\alpha,G_*}^{m^{*}}\rbrace,\ $$
$\text{for}\ u \in W^{\alpha,G_*}(\mathbb{R}^{N}).$
\end{lemma}

\begin{lemma}\label{lem777}[see \cite{29} Lemma 3.6]\\
$\text{For}\ u \in W^{\alpha,G}(\mathbb{R}^{N})$ the following properties hold true:
\begin{enumerate}
\item If $\Vert u\Vert >1,$ then $\Vert u\Vert^{L}\leq\tilde{\rho}(u)\leq \Vert u\Vert^{m}.$
\item If $\Vert u\Vert\leq 1,$ then $\Vert u\Vert^{m}\leq\tilde{\rho}(u)\leq \Vert u\Vert^{L}.$
\item From $(1)$ and $(2),$ we get
\begin{equation}\label{eq777}
\min\lbrace \Vert u\Vert^{L},\Vert u\Vert^{m}\rbrace\leq \tilde{\rho}(u)\leq \max\lbrace \Vert u\Vert^{L},\Vert u\Vert^{m}\rbrace.
\end{equation}
\end{enumerate}
\end{lemma}
We make the following assumptions on functions $f$ and $K$.\\
$\bullet$ The function $f:\ \mathbb{R}\longrightarrow\mathbb{R}$ is of class $C^{1}$ and satisfy the following growth conditions at the origin and  infinity
\begin{enumerate}
\item[$(f_1)$]\ \ \ \
$\displaystyle{\lim_{\vert t\vert \rightarrow 0^{+}}\frac{f(t)}{g(t)}=0}.$
\item[$(f_2)$]\ \ \ \
$\displaystyle{\lim_{\vert t\vert \rightarrow +\infty }\frac{f(t)}{g_{*}(t)}=0}.$
\item[$(f_3)$]\ \ \ \
$\displaystyle{\text{There is}\  \theta \in]2m,L^{*}[\  \text{such that}\ 0<\theta F(t)\leq tf(t),\ \text{for all}\ \vert t\vert >0},$\\
where $\displaystyle{F(t)=\int_{0}^{t}f(s)\ ds}.$
\item[$(f_4)$]\ \ \ \
 $\displaystyle{\text{The map}\  t\mapsto\frac{f(t)}{ \vert t\vert^{2m-1}},\  \ \text{is increasing}.}$
\end{enumerate}
$\bullet$ The function $K:\ \mathbb{R}^{N}\longrightarrow\mathbb{R}$  is continuous on $\mathbb{R}^{N}$ and satisfy  the following assumptions:
\begin{enumerate}
\item[$(K_1)$]\ \ \ \
 $K(x)>0,\ \ \text{for all} \ x\in \mathbb{R}^{N}\ \text{and}\ K\ \in L^{\infty}(\mathbb{R}^{N}).$
\item[$(K_2)$]\ \ \ \
$\text{If}\ (A_{n})_{n} \subset \mathbb{R}^{N}\  \text{is a sequence of Borel sets such that the Lebesgue measure}\ $\\
 $mes(A_{n})\leq R,\ \text{for all}\  n \in \mathbb{N}\  \text{and some}\  R>0,\ \text{then}$
$$\displaystyle{\lim_{r\rightarrow +\infty}\int_{A_{n}\cap B_{r}^{c}(0)}K(x)\ dx =0}, \ \ \text{uniformly in} \  n \in \mathbb{N}.$$
\end{enumerate}
As a consequence of the above assumptions, we can deduce the
following lemma which will be useful in the sequel.
\begin{lemma}\label{lem2.8}Assume that  $(f_4)$ and $(g_{1})-(g_{3})$  hold, then
\begin{enumerate}
\item[$(a)$] the function
\begin{equation}\label{3.40}
t\longmapsto G(t)-\frac{1}{2m}g(t)t\ \ \text{is  even and
increasing, for all}\  t>0.
\end{equation}

\item[$(b)$] the function
\begin{equation}\label{3.42}
t\longmapsto \frac{1}{2m}f(t)t-F(t)\ \ \text{is increasing on}\
]0,+\infty[\  \text{and decreasing on}\ \ ]-\infty,0[.
\end{equation}
\item[$(c)$] the function
\begin{equation}\label{3.43}
t\longmapsto \frac{1}{\displaystyle{\int_{\mathbb{R}^{N}}G(tu(x))\
dx}}\ \ \text{is decreasing on }\ ]0,+\infty[,\ \text{for fixed }\
u\in W^{\alpha,G}(\mathbb{R}^{N})\setminus\lbrace 0\rbrace.
\end{equation}

\end{enumerate}
\end{lemma}
\begin{proof}
See \cite{6}.
\end{proof}

Our main result asserts that problem \eqref{eqn:einstein} has at
least one sign-changing solution. More precisely:
\begin{thm}\label{thm1}
Suppose that assumptions $(f_1)-(f_4)$, $(g_1)-(g_3)$ and
$(K_1)-(K_2)$ are fulfilled, then, problem \eqref{eqn:einstein}
admits at least one nodal nontrivial weak solution.
\end{thm}
\section{Compactly embedding}
In this section, we are going to prove a  compactness embedding
result. Let $M$ be an $N$-function, we define the weighted Orlicz
space

$$\displaystyle{L^{M}_{K}(\mathbb{R}^{N})=\bigg{\{} u:\ \mathbb{R}^{N}\longrightarrow \mathbb{R}, \ u\ \text{mesurable and} \int_{\mathbb{R}^{N}}K(x)M(\vert u\vert)\ dx <\infty \bigg{\}}}.$$
\begin{prop}\label{pro1}
Suppose that the assumptions  $(g_1)-(g_3)$ and $(K_{1})-(K_{2})$
are fulfilled.  Let $M$ be an $N$-function such that
\begin{equation}\label{h1}
\lim_{\vert t\vert\rightarrow 0^{+}}\frac{M(t)}{G(\lambda t)}=0, \ \ \text{for all}\ \ \lambda>0
 \end{equation}
  and
\begin{equation}\label{h2}
   G_{*} \ \text{is essentially stronger than}\  M\  (M\prec\prec G_{*}),
\end{equation}
 then, the embedding $W^{\alpha,G}(\mathbb{R}^{N})\hookrightarrow L^{M}_{K}(\mathbb{R}^{N})$ is compact.
\end{prop}
\begin{proof}[Proof]
Since
$$\begin{array}{rlll}
\displaystyle{\lim_{\vert t\vert\rightarrow +\infty}\frac{M(t)}{G_{*}(\lambda t)}=0} &  and  & \displaystyle{\lim_{\vert t\vert\rightarrow 0^{+}}\frac{M(t)}{G(\lambda t)}=0,}
& \text{for all}\ \ \lambda>0
\end{array}$$
then, for fixing $ \varepsilon>0,$ there exist $0<t_{0}^{\lambda}(\varepsilon)<t_{1}^{\lambda}(\varepsilon)$ and $C(\varepsilon,\lambda)>0$, such that
\begin{equation}\label{eqr1}
M(t)\leq \varepsilon \left[ G( \lambda t)+G_{*}( \lambda t)\right] +C(\varepsilon,\lambda)\chi_{[t_{0}^{\lambda}(\varepsilon),t_{1}^{\lambda}(\varepsilon)]}( t)G_{*}( \lambda t),
\ \forall\ t \in \mathbb{R}.
\end{equation}
Let $(u_n)_{n\in \mathbb{N}}$ be a bounded sequence in $W^{\alpha,G}(\mathbb{R}^{N})$. Since $W^{\alpha,G}(\mathbb{R}^{N})$ is reflexive, up to subsequence,
 there exists $u\in W^{\alpha,G}(\mathbb{R}^{N})$ such that $u_n\rightharpoonup u\ \ \text{in}\ W^{\alpha,G}(\mathbb{R}^{N}).$ \
Put $v_n=u_n-u$. Then,  $v_n\rightharpoonup 0\ \ \text{in}\
W^{\alpha,G}(\mathbb{R}^{N})$.
 Using $(\ref{eqr1})$, assumption $(K_1)$ and integrating over $B^{c}_{r}(0)$ where $r>0$, we have
\begin{equation}\label{eq147}
\displaystyle{\int_{B^{c}_{r}(0)}K(x)M( v_{n})\ dx\leq\ \ \varepsilon\Vert K\Vert_{\infty}\mathcal{Q}(v_{n})  + C(\varepsilon,\lambda) G_{*}( \lambda t_{1}^{\lambda}(\varepsilon))\int_{A_{n}\cap B^{c}_{r}(0)}K(x)\ dx,\ \text{for all}\ n\in\mathbb{N}},
\end{equation}
where
$$\mathcal{Q}(v_{n}):=\left[ \int_{B^{c}_{r}(0)}G( \lambda v_{n})\ dx+\int_{B^{c}_{r}(0)}G_{*}( \lambda v_{n})\ dx \right] $$
and
 $$
 \displaystyle{A_{n}=\left\lbrace  x\in \mathbb{R}^{N}:\ t_{0}^{\lambda}(\varepsilon)\leq \vert v_{n}(x)\vert \leq t_{1}^{\lambda}(\varepsilon)\right\rbrace.}
$$
Evidently, the sequence  $(v_{n})_n$  is bounded in
$W^{\alpha,G}(\mathbb{R}^{N})$. Then,  applying Theorem \ref{thm2}
and Lemma \ref{lem2.6} , there is $C>0$ such that
$$\int_{\mathbb{R}^{N}}G( \lambda v_{n})\ dx\leq C\ \  \text{and}\ \ \displaystyle{\int_{\mathbb{R}^{N}}G_{*}( \lambda u_{n})\ dx\leq C}\ \ \text{for all}\ \ n\in \mathbb{N},$$
thus the sequence $\lbrace \mathcal{Q}(u_{n})\rbrace$ is bounded.\\
We observe that  $mes(A_{n})<+\infty$ for all $n\in\mathbb{N}$, indeed:
$$G_{*}( \lambda t_{0}^{\lambda}(\varepsilon))mes(A_{n})\leq \int_{A_{n}}G_{*}( \lambda v_{n}(x))\ dx \leq C\ \text{for all}\ n\in\mathbb{N}.$$
By  assumption $(K_{2}),$ for all $\varepsilon^{'}>0$, there is  $R_{\varepsilon^{'}}>0$, such that

\begin{equation}\label{eq258}
\int_{A_{n}\cap B^{c}_{r}(0)}K(x)dx \leq \varepsilon^{'}\ \ \text{for all}\ r\geq R_{\varepsilon^{'}}\ \text{and}\  n\in\mathbb{N}.
\end{equation}
Using $(\ref{eq147})$, and  choosing $\displaystyle{\varepsilon^{'}=\frac{\varepsilon}{C(\varepsilon,\lambda) G_{*}( \lambda t_{1}^{\lambda}(\varepsilon))}}$ in $(\ref{eq258})$, we can see that
\begin{equation}\label{eq369}
\int_{B_{r}^{c}(0)}K(x)M( v_{n})\ dx\leq  \varepsilon \left[ M +1\right],
 \ \text{for all}\ r\geq R_{\varepsilon^{'}}\ \text{and}\  n\in\mathbb{N}
\end{equation}
where $M=2C\Vert K\Vert_{\infty}$.\\
In the other side, from Theorem \ref{thm2} and the fact that  $K$ is a continuous function, we get
\begin{equation}\label{eq753}
\displaystyle{\lim_{n\rightarrow +\infty}\int_{B_{r}(0)}K(x)M(
v_{n})\ dx =0}, \ \ \forall r\geq R_{\epsilon^{'
}}.
\end{equation}
Putting together $(\ref{eq369})$ and $(\ref{eq753})$, we find that
$$v_n=u_n-u\rightarrow 0\ \text{in}\ L^{M}_{K}(\mathbb{R}^{N})$$
thus
$u_{n}\longrightarrow u$ in $L^{M}_{K}(\mathbb{R}^{N})$.
\end{proof}

Now, we  prove a compactness embedding results related to the
non-linear term.
\begin{lemma}\label{lem222}
 Suppose that  $(f_{1})-(f_{2})$, $(g_{1})-(g_{3})$  and  $(K_{1})-(K_{2})$ hold. Let $(u_{n})_n$ a bounded sequence in $W^{\alpha,G}(\mathbb{R}^{N})$, then
 \begin{enumerate}

 \item[$(1)$] $\displaystyle{\lim_{n\rightarrow +\infty}\int_{\mathbb{R}^{N}}K(x)f(u_{n})u_{n}\ dx =\int_{\mathbb{R}^{N}}K(x)f(u)u\ dx};$
\item[$(2)$] $\displaystyle{\lim_{n\rightarrow +\infty}\int_{\mathbb{R}^{N}}K(x)F(u_{n})\ dx =\int_{\mathbb{R}^{N}}K(x)F(u)\ dx};$
\item[$(3)$] $\displaystyle{\lim_{n\rightarrow +\infty}\int_{\mathbb{R}^{N}}K(x)f(u^{\pm}_{n})u^{\pm}_{n}\ dx =\int_{\mathbb{R}^{N}}K(x)f(u^{\pm})u^{\pm}\ dx,}$\\

  where $u^{+}:=\max\lbrace0,u\rbrace$, $u^{-}:=\min\lbrace0,u\rbrace$.

 \end{enumerate}
\end{lemma}
\begin{proof}[Proof]
 $(1)$ Let  $M$ be an $N$-function satisfying $(\ref{h1})$ and $(\ref{h2})$.
 By $(f_{1})$-$(f_{2})$, $(g_2)$ and Lemma \ref{lem2.5},  for all $\varepsilon>0$, there exists $C(\varepsilon)>0$ such that
\begin{equation}\label{eq159}
K(x)f(t)t\leq \varepsilon\Vert K\Vert_{\infty}\left[ mG(t)+m^{*}G_{*}(t)\right] +C(\varepsilon) K(x)M(t), \ \ \text{for all}\ \ t\in\mathbb{R}.
\end{equation}
Since $W^{\alpha,G}(\mathbb{R}^{N})$ is reflexive,
$u_{n}\rightharpoonup u$ in $W^{\alpha,G}(\mathbb{R}^{N})$. Similar
to the proof of   $(\ref{eq369})$, for fixing $\varepsilon>0$, we
have
\begin{align}\label{eq789}
& \int_{B_{r}^{c}(0)}K(x)M(u_{n})\ dx\leq\frac{\varepsilon}{C(\varepsilon)},\ \ \text{for all}\ \ n\in\mathbb{N}\\
& \text{and}\nonumber \\
& \int_{B_{r}^{c}(0)}K(x)M(u)\ dx\leq\varepsilon,\ \ \text{for all}\ r\geq R^{'} _{\varepsilon}.\nonumber
\end{align}
Exploiting Theorem \ref{thm2} and Lemma \ref{lem2.5}, there exists a positive constant $C$ such that
\begin{equation}\label{eq123}
\int_{\mathbb{R}^{N}}G(u_{n})\ dx \leq C\ \ \text{and}\ \  \int_{\mathbb{R}^{N}}G_{*}(u_{n})\ dx \leq C.
\end{equation}
Combining $(\ref{eq159})$, $(\ref{eq789})$ and $(\ref{eq123})$, we
infer that
\begin{align}\label{et123}
& \int_{B_{r}^{c}(0)}K(x)f(u_{n})u_{n}\ dx\leq(2m^{*}\Vert K\Vert_{\infty}C+1)\varepsilon, \ \ \text{for all}\ \ n\in\mathbb{N}\\
& \text{and}\nonumber\\
& \int_{B_{r}^{c}(0)}K(x)f(u)u\ dx\leq(2m^{*}\Vert
K\Vert_{\infty}C+1)\varepsilon.\nonumber
\end{align}
By applying  Strauss's compactness Lemma [\cite{100}, Theorem A.I, p. 338], we get
\begin{equation}\label{et321}
\lim\limits_{n\rightarrow+\infty}\int_{B_{r}(0)}K(x)f(u_{n})u_{n}\ dx=\int_{B_{r}(0)}K(x)f(u)u\ dx
\end{equation}
Using $(\ref{et123})$ and $(\ref{et321})$, we find that

$$\left| \int_{\mathbb{R}^{N}}K(x)f(u_{n})u_{n}\ dx -\int_{\mathbb{R}^{N}}K(x)f(u)u\ dx\right|\leq (2m^{*}\Vert K\Vert_{\infty}C+1)\varepsilon,\ \text{for all}\ n\in \mathbb{N},$$
thus the proof of $(1)$.\\
 $(2)$ By $(f_{3})$ and $(\ref{eq159})$,  $\text{for all}\  \varepsilon >0$ there exist $C_{\varepsilon}>0$, such that
\begin{align}\label{eq761}
\theta F(t)  \leq f(t)t
 \leq \varepsilon\left[ mG(t)+m^{*}G_{*}(t)\right] +C_{\varepsilon} M(t), \ \ \text{for all}\ \ t\in\mathbb{R}.
\end{align}
Using $(\ref{eq761})$ and arguing as in the first case, we deduce the desired result.\\

  $(3)$  Since  $\vert t^{\pm}-s^{\pm}\vert\leq \vert t-s\vert$
for all $t,s\in \mathbb{R},$ the proof is similar to the first case.
\end{proof}

\section{Technical lemmas and proof of main result}
We divide this section into two parts. In the first, we give some
technical lemmas. In the second part, using the results obtained in
the first subsection, we prove Theorem \ref{thm1}.\\

Again, we recall that we look for sign-changing weak solutions of
problem $(\ref{eqn:einstein})$, that is a function $u\in
W^{\alpha,G}(\mathbb{R}^{N})$ such that $u^{+}:=\max\lbrace
u,0\rbrace\neq 0$, $u^{-}:=\min\lbrace u,0\rbrace\neq 0$ in
$\mathbb{R}^{N}$ and
\begin{align*}
&\int_{\mathbb{R}^{N}}\int_{\mathbb{R}^{N}}g\left( \frac{u(x)-u(y)}{\vert x-y\vert^{\alpha}}\right)\frac{v(x)-v(y)}{\vert x-y\vert^{\alpha+N}}dxdy\\
&+\int_{\mathbb{R}^{N}}g(u)v\ dx\\
&=\int_{\mathbb{R}^{N}}K(x)f(u)v\ dx,\ \text{for all}\ v \in W^{\alpha,G}(\mathbb{R}^{N}).
\end{align*}
\\

Let the functional $J:\ W^{\alpha,G}(\mathbb{R}^{N})\rightarrow \mathbb{R}$ defined by
$$J(u):=\int_{\mathbb{R}^{N}}\int_{\mathbb{R}^{N}}G\left( \frac{u(x)-u(y)}{\vert x-y\vert^{\alpha}}\right)\frac{dxdy}{\vert x-y\vert^{N}} + \int_{\mathbb{R}^{N}}G(u)\ dx-\int_{\mathbb{R}^{N}}K(x)F(u)\ dx .$$
In view of assumptions on $K$ and $f$, we see that $J$ is Fr\`echet differentiable and
\begin{align*}
\langle J^{'}(u),v\rangle & =\int_{\mathbb{R}^{N}}\int_{\mathbb{R}^{N}}g\left( \frac{u(x)-u(y)}{\vert x-y\vert^{\alpha}}\right)\frac{v(x)-v(y)}{\vert x-y\vert^{\alpha+N}}dxdy+\int_{\mathbb{R}^{N}}g(u)v\ dx\\
& -\int_{\mathbb{R}^{N}}K(x)f(u)v\ dx,\ \ \text{for all}\ v \in W^{\alpha,G}(\mathbb{R}^{N}).
\end{align*}
For the proof see [\cite{5}-Proposition 2.1 and \cite{28}-Proposition 4.1].\\
 The weak solutions of problem $(\ref{eqn:einstein})$ are the critical points of $J.$ Let $\mathcal{N}$  be the Nehari manifold  and $\mathcal{M}$ the nodal Nehari manifold associated to $J$:
$$\mathcal{N}:=\lbrace u \in W^{\alpha,G}(\mathbb{R}^{N})\backslash\lbrace 0\rbrace :\ \langle J^{'}(u),u\rangle =0 \rbrace ;$$
$$\mathcal{M}:=\lbrace w\in \mathcal{N}:\ w^{+}\neq 0,\ w^{-}\neq 0,\ \langle J^{'}(w),w^{+}\rangle=\ 0\ = \langle J^{'}(w),w^{-}\rangle\rbrace .$$
Note that $\mathcal{M}\neq \emptyset$ (will be treated later in Lemma \ref{lem3.3} ).\\

We look for a least energy sign-changing weak solution of problem (P), it means to look for function $w \in \mathcal{M}$ such that
$$J(w)=\inf_{v\in\mathcal{M}}J(v).$$
 Let $w\in \mathcal{N}$,
$ A:= \lbrace x \in\mathbb{R}^{N};\ w(x)\geq 0\rbrace$ and  $A^{c}=B:= \lbrace x \in\mathbb{R}^{N};\ w(x)< 0\rbrace$.\\
 For all $(t,s)\in [0,+\infty)\times[0,+\infty),$ we denote by
\begin{align*}
E_{t,s}: & =\int_{B}\int_{A}g\left( \frac{tw^{+}(x)-sw^{-}(y)}{\vert x-y\vert^{\alpha}}\right)\frac{tw^{+}(x)}{\vert x-y\vert^{\alpha +N}}dxdy\\
& - \int_{B}\int_{A}g\left( \frac{tw^{+}(x)}{\vert x-y\vert^{\alpha}}\right)\frac{tw^{+}(x)}{\vert x-y\vert^{\alpha +N}}dxdy\\
& + \int_{A}\int_{B}g\left( \frac{sw^{-}(x)-tw^{+}(y)}{\vert x-y\vert^{\alpha}}\right)\frac{-tw^{+}(y)}{\vert x-y\vert^{\alpha +N}}dxdy\\
& - \int_{A}\int_{B}g\left( \frac{-tw^{+}(y)}{\vert x-y\vert^{\alpha}}\right)\frac{-tw^{+}(y)}{\vert x-y\vert^{\alpha +N}}dxdy,
\end{align*}

\begin{align*}
H_{t,s}: & = \int_{B}\int_{A}g\left( \frac{tw^{+}(x)-sw^{-}(y)}{\vert x-y\vert^{\alpha}}\right)\frac{-sw^{-}(y)}{\vert x-y\vert^{\alpha +N}}dxdy\\
& - \int_{B}\int_{A}g\left( \frac{-sw^{-}(y)}{\vert x-y\vert^{\alpha}}\right)\frac{-sw^{-}(y)}{\vert x-y\vert^{\alpha +N}}dxdy\\
& + \int_{A}\int_{B}g\left( \frac{sw^{-}(x)-tw^{+}(y)}{\vert x-y\vert^{\alpha}}\right)\frac{sw^{-}(x)}{\vert x-y\vert^{\alpha +N}}dxdy\\
& - \int_{A}\int_{B}g\left( \frac{sw^{-}(x)}{\vert x-y\vert^{\alpha}}\right)\frac{sw^{-}(x)}{\vert x-y\vert^{\alpha +N}}dxdy,
\end{align*}
and
\begin{align*}
Q_{t,s} & = \int_{B}\int_{A}G\left( \frac{tw^{+}(x)-sw^{-}(y)}{\vert  x-y\vert^{\alpha}}\right)\frac{dxdy}{\vert x-y\vert^{N}}\\
& + \int_{A}\int_{B}G\left( \frac{sw^{-}(x)-tw^{+}(y)}{\vert  x-y\vert^{\alpha}}\right)\frac{dxdy}{\vert x-y\vert^{N}}\\
& -\int_{B}\int_{A}G\left( \frac{tw^{+}(x)}{\vert  x-y\vert^{\alpha}}\right)\frac{dxdy}{\vert x-y\vert^{N}}
 -\int_{B}\int_{A}G\left( \frac{-sw^{-}(y)}{\vert  x-y\vert^{\alpha}}\right)\frac{dxdy}{\vert x-y\vert^{N}}\\
& - \int_{A}\int_{B}G\left( \frac{sw^{-}(x)}{\vert  x-y\vert^{\alpha}}\right)\frac{dxdy}{\vert x-y\vert^{N}}
 -\int_{A}\int_{B}G\left( \frac{-tw^{+}(y)}{\vert  x-y\vert^{\alpha}}\right)\frac{dxdy}{\vert x-y\vert^{N}}.
\end{align*}
\begin{rem}\label{rem1}\ \\
\begin{enumerate}
\item[$(1)$] Using Lemma \ref{lem444} and taking into account  the fact that $G(t)$  is an even and increasing function, we see that $Q_{t,s}$ is positive, for all $(t,s)\in [0,+\infty)\times[0,+\infty)$.
\item[$(2)$]By using the fact that $g(t)$ is an odd and  nondecreasing function, we prove that $E_{t,s}$ and $H_{t,s}$ are positives, for all $(t,s)\in [0,+\infty)\times[0,+\infty)$.
\end{enumerate}
\end{rem}

In the following, we give some technical Lemmas which will be useful later. 
\subsection{Technical Lemmas}\ \\
In this subsection we give some technical lemmas which will be
useful in the proof of our main result.
\begin{lemma}\label{lem100}
Suppose that assumptions $(f_1)-(f_4)$, $(g_1)-(g_3)$ and $(K_1)$
are fulfilled. Let $w\in \mathcal{N},$ then, for all $(t,s)\in
[0,+\infty)\times[0,+\infty),$ we have
\begin{enumerate}
\item[$(i)$]$\langle J^{'}(tw^{+}+sw^{-}),tw^{+}\rangle=\langle J^{'}(tw^{+}),tw^{+}\rangle + E_{t,s}$,
\item[$(ii)$]$\langle J^{'}(tw^{+}+sw^{-}),sw^{-}\rangle=\langle J^{'}(sw^{-}),sw^{-}\rangle + H_{t,s}$,
\item[$(iii)$] $J(tw^{+}+sw^{-})=J(tw^{+})+J(sw^{-})+Q_{t,s}.$
\end{enumerate}

\end{lemma}
\begin{proof}[Proof]
$(i)$ Let $(t,s)\in[0,+\infty)\times[0,+\infty),$ then
\begin{align*}
& \langle J^{'}(tw^{+}+sw^{-}),tw^{+}\rangle\\
  & =\int_{\mathbb{R}^{N}}\int_{\mathbb{R}^{N}}g\left( \frac{tw^{+}(x)-tw^{+}(y)+sw^{-}(x)-sw^{-}(y)}{\vert x-y\vert^{\alpha}}\right)\frac{tw^{+}(x)-tw^{+}(y)}{\vert x-y\vert^{\alpha +N}}dxdy\\
\ & +\int_{\mathbb{R}^{N}}g(tw^{+})tw^{+}\ dx-\int_{\mathbb{R}^{N}}K(x)f(tw^{+})tw^{+}\ dx\\
\ & =\int_{A}\int_{A}g\left( \frac{tw^{+}(x)-tw^{+}(y)}{\vert x-y\vert^{\alpha}}\right)\frac{tw^{+}(x)-tw^{+}(y)}{\vert x-y\vert^{\alpha +N}}dxdy\\
& + \int_{B}\int_{A}g\left( \frac{tw^{+}(x)-sw^{-}(y)}{\vert x-y\vert^{\alpha}}\right)\frac{tw^{+}(x)}{\vert x-y\vert^{\alpha +N}}dxdy\\
\ &  + \int_{A}\int_{B}g\left( \frac{sw^{-}(x)-tw^{+}(y)}{\vert x-y\vert^{\alpha}}\right)\frac{-tw^{+}(y)}{\vert x-y\vert^{\alpha +N}}dxdy \\
\ &  +\int_{\mathbb{R}^{N}}g(tw^{+})tw^{+}\
dx-\int_{\mathbb{R}^{N}}K(x)f(tw^{+})tw^{+}\ dx.
\end{align*}
On the other hand
\begin{align*}
 \langle J^{'}(tw^{+}),tw^{+}\rangle & = \int_{\mathbb{R}^{N}}\int_{\mathbb{R}^{N}}g\left( \frac{tw^{+}(x)-tw^{+}(y)}{\vert x-y\vert^{\alpha}}\right)\frac{tw^{+}(x)-tw^{+}(y)}{\vert x-y\vert^{\alpha +N}}dxdy\\
 & + \int_{\mathbb{R}^{N}}g(tw^{+})tw^{+}\ dx-\int_{\mathbb{R}^{N}}K(x)f(tw^{+})tw^{+}\ dx\\
 & = \int_{A}\int_{A}g\left( \frac{tw^{+}(x)-tw^{+}(y)}{\vert x-y\vert^{\alpha}}\right)\frac{tw^{+}(x)-tw^{+}(y)}{\vert x-y\vert^{\alpha +N}}dxdy \\
 &  + \int_{B}\int_{A}g\left( \frac{tw^{+}(x)}{\vert x-y\vert^{\alpha}}\right)\frac{tw^{+}(x)}{\vert x-y\vert^{\alpha +N}}dxdy\\
 &  + \int_{A}\int_{B}g\left( \frac{-tw^{+}(y)}{\vert x-y\vert^{\alpha}}\right)\frac{-tw^{+}(y)}{\vert x-y\vert^{\alpha +N}}dxdy\\
\ &  +\int_{\mathbb{R}^{N}}g(tw^{+})tw^{+}\
dx-\int_{\mathbb{R}^{N}}K(x)f(tw^{+})tw^{+}\ dx.
\end{align*}
Combining the above pieces of informations, we conclude that
$$\langle J^{'}(tw^{+}+sw^{-}),tw^{+}\rangle=\langle J^{'}(tw^{+}),tw^{+}\rangle + E_{t,s},\ \text{for all} \ (t,s)\in [0,+\infty)\times[0,+\infty).$$
Thus the proof of $(i)$. The proofs of $(ii)$ and $(iii)$ are
similar to that in $(i)$.
\end{proof}
\begin{lemma}\label{lem400}
Suppose that assumptions $(f_1)-(f_4)$, $(g_1)-(g_3)$ and $(K_1)$
are fulfilled. Let $w\in \mathcal{N},$ then, for all $(t,s)\in
[0,+\infty)\times[0,+\infty),$ we have
\begin{enumerate}
\item[$(i)$] $\langle J^{'}(tw^{+}+sw^{-}),tw^{+}\rangle\geq\langle J^{'}(tw^{+}),tw^{+}\rangle,$
\item[$(ii)$] $\langle J^{'}(tw^{+}+sw^{-}),sw^{-}\rangle\geq\langle J^{'}(sw^{-}),sw^{-}\rangle,$
\item[$(iii)$] $J(tw^{+}+sw^{-})\geq J(tw^{+})+J(sw^{-}).$
\end{enumerate}
\end{lemma}
\begin{proof}[Proof]
Combining Lemma \ref{lem100} and Remark \ref{rem1}, we get the
desired results.
\end{proof}
\begin{lemma}\label{lemgrand}
Suppose that assumptions $(f_1)-(f_4)$, $(g_1)-(g_3)$ and $(K_1)$
are fulfilled. Let $w\in\mathcal{N}$, then, for all $(s,t)\in
[0,1]\times[0,1]\backslash\lbrace (1,1)\rbrace$, we have
$$J(tw^{+}+sw^{-})-\frac{1}{2m}\langle J^{'}(tw^{+}+sw^{-}),tw^{+}+sw^{-}\rangle < J(w).$$
\end{lemma}
\begin{proof}[Proof]
 Let $w\in\mathcal{N}$ and $(s,t)\in [0,1]\times[0,1]\backslash\lbrace (1,1)\rbrace$
\begin{align}\label{11}
& J(tw^{+}+sw^{-})-\frac{1}{2m}\langle J^{'}(tw^{+}+sw^{-}),tw^{+}+sw^{-}\rangle\nonumber\\
& =\int_{\mathbb{R}^{N}}\int_{\mathbb{R}^{N}}G\left( \frac{tw^{+}(x)+sw^{-}(x)-tw^{+}(y)-sw^{-}(y)}{\vert x-y\vert^{\alpha}}\right)\frac{dxdy}{\vert x-y\vert^{N}}\nonumber\\
&-\frac{1}{2m}\int_{\mathbb{R}^{N}}\int_{\mathbb{R}^{N}}g\left( \frac{tw^{+}(x)+sw^{-}(x)-tw^{+}(y)-sw^{-}(y)}{\vert x-y\vert^{\alpha}}\right)\frac{tw^{+}(x)+sw^{-}(x)-tw^{+}(y)-sw^{-}(y)}{\vert x-y\vert^{\alpha +N}}dxdy\nonumber\\
& + \int_{\mathbb{R}^{N}}\left[ G(tw^{+}+sw^{-})-\frac{1}{2m}g(tw^{+}+sw^{-})(tw^{+}+sw^{-})\right] dx\nonumber\\
& + \int_{\mathbb{R}^{N}}\left[\frac{1}{2m}f(tw^{+}+sw^{-})(tw^{+}+sw^{-})-F(tw^{+}+sw^{-})\right] dx\nonumber\\
& = \int_{A}\int_{A}G\left( \frac{tw^{+}(x)-tw^{+}(y)}{\vert x-y\vert^{\alpha}}\right)\frac{dxdy}{\vert x-y\vert^{N}}\nonumber\\
&-\frac{1}{2m}\int_{A}\int_{A}g\left( \frac{tw^{+}(x)-tw^{+}(y)}{\vert x-y\vert^{\alpha}}\right)\frac{tw^{+}(x)-tw^{+}(y)}{\vert x-y\vert^{\alpha +N}}dxdy\nonumber\\
& + \int_{B}\int_{B}G\left( \frac{sw^{-}(x)-sw^{-}(y)}{\vert x-y\vert^{\alpha}}\right)\frac{dxdy}{\vert x-y\vert^{N}}\nonumber\\
&-\frac{1}{2m}\int_{B}\int_{B}g\left( \frac{sw^{-}(x)-sw^{-}(y)}{\vert x-y\vert^{\alpha}}\right)\frac{sw^{-}(x)-sw^{-}(y)}{\vert x-y\vert^{\alpha +N}}dxdy\nonumber\\
&+ \int_{B}\int_{A}G\left( \frac{tw^{+}(x)-sw^{-}(y)}{\vert x-y\vert^{\alpha}}\right)\frac{dxdy}{\vert x-y\vert^{N}}\nonumber\\
&-\frac{1}{2m}\int_{B}\int_{A}g\left( \frac{tw^{+}(x)-sw^{-}(y)}{\vert x-y\vert^{\alpha}}\right)\frac{tw^{+}(x)-sw^{-}(y)}{\vert x-y\vert^{\alpha +N}}dxdy\nonumber\\
& +\int_{A}\int_{B}G\left( \frac{sw^{-}(x)-tw^{+}(y)}{\vert x-y\vert^{\alpha}}\right)\frac{dxdy}{\vert x-y\vert^{N}}\nonumber\\
&-\frac{1}{2m}\int_{A}\int_{B}g\left( \frac{sw^{-}(x)-tw^{+}(y)}{\vert x-y\vert^{\alpha}}\right)\frac{sw^{-}(x)-tw^{+}(y)}{\vert x-y\vert^{\alpha +N}}dxdy\nonumber\\
& + \int_{\mathbb{R}^{N}}\left[
G(tw^{+})-\frac{1}{2m}g(tw^{+})(tw^{+})\right] dx
 + \int_{\mathbb{R}^{N}}\left[\frac{1}{2m}f(tw^{+})(tw^{+})-F(tw^{+})\right] dx\nonumber\\
 & + \int_{\mathbb{R}^{N}}\left[ G(sw^{-})-\frac{1}{2m}g(sw^{-})(sw^{-})\right] dx
 + \int_{\mathbb{R}^{N}}\left[\frac{1}{2m}f(sw^{-})(sw^{-})-F(sw^{-})\right] dx.\nonumber\\
\end{align}
In the other hand, since $w\in\mathcal{N}$, one has
\begin{align}\label{22}
&J(w)= J(w)-\frac{1}{2m}\langle J^{'}(w),w\rangle\nonumber\\
& =\int_{\mathbb{R}^{N}}\int_{\mathbb{R}^{N}}G\left( \frac{w^{+}(x)+w^{-}(x)-w^{+}(y)-w^{-}(y)}{\vert x-y\vert^{\alpha}}\right)\frac{dxdy}{\vert x-y\vert^{N}}\nonumber\\
&-\frac{1}{2m}\int_{\mathbb{R}^{N}}\int_{\mathbb{R}^{N}}g\left( \frac{w^{+}(x)+w^{-}(x)-w^{+}(y)-w^{-}(y)}{\vert x-y\vert^{\alpha}}\right)\frac{w^{+}(x)+w^{-}(x)-w^{+}(y)-w^{-}(y)}{\vert x-y\vert^{\alpha +N}}dxdy\nonumber\\
& + \int_{\mathbb{R}^{N}}\left[ G(w^{+}+w^{-})-\frac{1}{2m}g(w^{+}+w^{-})(w^{+}+w^{-})\right] dx\nonumber\\
& + \int_{\mathbb{R}^{N}}\left[\frac{1}{2m}f(w^{+}+w^{-})(w^{+}+w^{-})-F(w^{+}+w^{-})\right] dx\nonumber\\
& = \int_{A}\int_{A}G\left( \frac{w^{+}(x)-w^{+}(y)}{\vert x-y\vert^{\alpha}}\right)\frac{dxdy}{\vert x-y\vert^{N}}\nonumber\\
&-\frac{1}{2m}\int_{A}\int_{A}g\left( \frac{w^{+}(x)-w^{+}(y)}{\vert x-y\vert^{\alpha}}\right)\frac{w^{+}(x)-w^{+}(y)}{\vert x-y\vert^{\alpha +N}}dxdy\nonumber\\
& + \int_{B}\int_{B}G\left( \frac{w^{-}(x)-w^{-}(y)}{\vert x-y\vert^{\alpha}}\right)\frac{dxdy}{\vert x-y\vert^{N}}\nonumber\\
&-\frac{1}{2m}\int_{B}\int_{B}g\left( \frac{w^{-}(x)-w^{-}(y)}{\vert x-y\vert^{\alpha}}\right)\frac{w^{-}(x)-w^{-}(y)}{\vert x-y\vert^{\alpha +N}}dxdy\nonumber\\
&+ \int_{B}\int_{A}G\left( \frac{w^{+}(x)-w^{-}(y)}{\vert x-y\vert^{\alpha}}\right)\frac{dxdy}{\vert x-y\vert^{N}}\nonumber\\
&-\frac{1}{2m}\int_{B}\int_{A}g\left( \frac{w^{+}(x)-w^{-}(y)}{\vert x-y\vert^{\alpha}}\right)\frac{w^{+}(x)-w^{-}(y)}{\vert x-y\vert^{\alpha +N}}dxdy\nonumber\\
& +\int_{A}\int_{B}G\left( \frac{w^{-}(x)-w^{+}(y)}{\vert x-y\vert^{\alpha}}\right)\frac{dxdy}{\vert x-y\vert^{N}}\nonumber\\
&-\frac{1}{2m}\int_{A}\int_{B}g\left( \frac{w^{-}(x)-w^{+}(y)}{\vert x-y\vert^{\alpha}}\right)\frac{w^{-}(x)-w^{+}(y)}{\vert x-y\vert^{\alpha +N}}dxdy\nonumber\\
& + \int_{\mathbb{R}^{N}}\left[
G(w^{+})-\frac{1}{2m}g(w^{+})(w^{+})\right] dx
 + \int_{\mathbb{R}^{N}}\left[\frac{1}{2m}f(w^{+})(w^{+})-F(w^{+})\right] dx\nonumber\\
 & + \int_{\mathbb{R}^{N}}\left[ G(w^{-})-\frac{1}{2m}g(w^{-})(w^{-})\right] dx
 + \int_{\mathbb{R}^{N}}\left[\frac{1}{2m}f(w^{-})(w^{-})-F(w^{-})\right] dx.\nonumber\\
\end{align}
Combining $(\ref{11})$ and $(\ref{22})$, we obtain
\begin{align*}
&J(w)-\frac{1}{2m}\langle J^{'}(w),w\rangle-J(tw^{+}+sw^{-})+\frac{1}{2m}\langle J^{'}(tw^{+}+sw^{-}),tw^{+}+sw^{-}\rangle\\
&= I_1+I_2+I_3+I_4+I_5,
\end{align*}
where
\begin{align*}
I_1 & =\int_{A}\int_{A}G\left( \frac{w^{+}(x)-w^{+}(y)}{\vert x-y\vert^{\alpha}}\right)\frac{dxdy}{\vert x-y\vert^{N}}\\
&-\frac{1}{2m}\int_{A}\int_{A}g\left( \frac{w^{+}(x)-w^{+}(y)}{\vert x-y\vert^{\alpha}}\right)\frac{w^{+}(x)-w^{+}(y)}{\vert x-y\vert^{\alpha +N}}dxdy\\
& -\int_{A}\int_{A}G\left( \frac{tw^{+}(x)-tw^{+}(y)}{\vert x-y\vert^{\alpha}}\right)\frac{dxdy}{\vert x-y\vert^{N}}\nonumber\\
&+\frac{1}{2m}\int_{A}\int_{A}g\left(
\frac{tw^{+}(x)-tw^{+}(y)}{\vert
x-y\vert^{\alpha}}\right)\frac{tw^{+}(x)-tw^{+}(y)}{\vert
x-y\vert^{\alpha +N}}dxdy,
\end{align*}
\begin{align*}
 I_2 & =\int_{B}\int_{B}G\left( \frac{w^{-}(x)-w^{-}(y)}{\vert x-y\vert^{\alpha}}\right)\frac{dxdy}{\vert x-y\vert^{N}}\\
&-\frac{1}{2m}\int_{B}\int_{B}g\left( \frac{w^{-}(x)-w^{-}(y)}{\vert x-y\vert^{\alpha}}\right)\frac{w^{-}(x)-w^{-}(y)}{\vert x-y\vert^{\alpha +N}}dxdy\\
&-\int_{B}\int_{B}G\left( \frac{sw^{-}(x)-sw^{-}(y)}{\vert x-y\vert^{\alpha}}\right)\frac{dxdy}{\vert x-y\vert^{N}}\\
&+\frac{1}{2m}\int_{B}\int_{B}g\left(
\frac{sw^{-}(x)-sw^{-}(y)}{\vert
x-y\vert^{\alpha}}\right)\frac{sw^{-}(x)-sw^{-}(y)}{\vert
x-y\vert^{\alpha +N}}dxdy,
\end{align*}
\begin{align*}
 I_3 & =\int_{B}\int_{A}G\left( \frac{w^{+}(x)-w^{-}(y)}{\vert x-y\vert^{\alpha}}\right)\frac{dxdy}{\vert x-y\vert^{N}}\\
&-\frac{1}{2m}\int_{B}\int_{A}g\left( \frac{w^{+}(x)-w^{-}(y)}{\vert x-y\vert^{\alpha}}\right)\frac{w^{+}(x)-w^{-}(y)}{\vert x-y\vert^{\alpha +N}}dxdy\\
& -\int_{B}\int_{A}G\left( \frac{tw^{+}(x)-sw^{-}(y)}{\vert x-y\vert^{\alpha}}\right)\frac{dxdy}{\vert x-y\vert^{N}}\\
&+\frac{1}{2m}\int_{B}\int_{A}g\left( \frac{tw^{+}(x)-sw^{-}(y)}{\vert x-y\vert^{\alpha}}\right)\frac{tw^{+}(x)-sw^{-}(y)}{\vert x-y\vert^{\alpha +N}}dxdy,\\
\end{align*}
\begin{align*}
 I_4 & =\int_{A}\int_{B}G\left( \frac{w^{-}(x)-w^{+}(y)}{\vert x-y\vert^{\alpha}}\right)\frac{dxdy}{\vert x-y\vert^{N}}\\
&-\frac{1}{2m}\int_{A}\int_{B}g\left( \frac{w^{-}(x)-w^{+}(y)}{\vert x-y\vert^{\alpha}}\right)\frac{w^{-}(x)-w^{+}(y)}{\vert x-y\vert^{\alpha +N}}dxdy\\
& -\int_{A}\int_{B}G\left( \frac{sw^{-}(x)-tw^{+}(y)}{\vert x-y\vert^{\alpha}}\right)\frac{dxdy}{\vert x-y\vert^{N}}\\
&+\frac{1}{2m}\int_{A}\int_{B}g\left( \frac{sw^{-}(x)-tw^{+}(y)}{\vert x-y\vert^{\alpha}}\right)\frac{sw^{-}(x)-tw^{+}(y)}{\vert x-y\vert^{\alpha +N}}dxdy\\
\end{align*}
and
\begin{align*}
 I_5 & =\int_{\mathbb{R}^{N}}\left[ G(w^{+})-\frac{1}{2m}g(w^{+})(w^{+})\right] dx
 + \int_{\mathbb{R}^{N}}\left[\frac{1}{2m}f(w^{+})(w^{+})-F(w^{+})\right] dx\\
 & + \int_{\mathbb{R}^{N}}\left[ G(w^{-})-\frac{1}{2m}g(w^{-})(w^{-})\right] dx
 + \int_{\mathbb{R}^{N}}\left[\frac{1}{2m}f(w^{-})(w^{-})-F(w^{-})\right] dx\\
 & -\int_{\mathbb{R}^{N}}\left[ G(tw^{+})-\frac{1}{2m}g(tw^{+})(tw^{+})\right] dx
 - \int_{\mathbb{R}^{N}}\left[\frac{1}{2m}f(tw^{+})(tw^{+})-F(tw^{+})\right] dx\\
 & - \int_{\mathbb{R}^{N}}\left[ G(sw^{-})-\frac{1}{2m}g(sw^{-})(sw^{-})\right] dx
 - \int_{\mathbb{R}^{N}}\left[\frac{1}{2m}f(sw^{-})(sw^{-})-F(sw^{-})\right] dx.
\end{align*}
By Lemma \ref{lem2.8}, it follows that
$$ I_1,\ I_2,\ I_3,\ I_4,\ I_5 > 0,$$
this gives the desired result.
\end{proof}
In the following lemma, we prove that $J$ is coercive on
$\mathcal{N}$ and in particular on $\mathcal{M}.$
\begin{lemma}\label{lem3.1}
Assume that the assumptions $(f_1)-(f_4)$, $(g_1)-(g_3)$  and $(K_1)-(K_2)$are fulfilled, then, we have
\begin{enumerate}
\item[(i)] $J_{\mid_{\mathcal{N}}}$ is coercive;
\item[(ii)] There exists $\varrho >0$  such that $\Vert u\Vert \geq \varrho$ for all $u\in\mathcal{N}$ and $\Vert w^{\pm}\Vert \geq \varrho$ for all $w\in\mathcal{M}.$
\end{enumerate}
\end{lemma}
\begin{proof}[Proof]

$(i)$ Let $u\in \mathcal{N}$. By Lemma \ref{lem777}, assumptions $(f_3)$ and $(g_2)$, we infer that
\begin{align*}
J(u) & =J(u)-\frac{1}{\theta}\langle J^{'}(u),u\rangle\\
 & \geq \int_{\mathbb{R}^{N}}\int_{\mathbb{R}^{N}}G\left( \frac{u(x)-u(y)}{\vert x-y\vert^{\alpha}}\right)\frac{dxdy}{\vert x-y\vert^{N}}+\int_{\mathbb{R}^{N}}G(u)\ dx \\
\ & -\frac{m}{\theta}\int_{\mathbb{R}^{N}}\int_{\mathbb{R}^{N}}G\left( \frac{u(x)-u(y)}{\vert x-y\vert^{\alpha}}\right)\frac{dxdy}{\vert x-y\vert^{N}}-\frac{m}{\theta}\int_{\mathbb{R}^{N}}G(u)\ dx\\
& +\frac{1}{\theta}\int_{\mathbb{R}^{N}}K(x)f(u)u\ dx - \int_{\mathbb{R}^{N}}K(x)F(u)\ dx\\
\ & \geq\left( 1-\frac{m}{\theta}\right) \left[ \int_{\mathbb{R}^{N}}\int_{\mathbb{R}^{N}}G\left( \frac{u(x)-u(y)}{\vert x-y\vert^{\alpha}}\right)\frac{dxdy}{\vert x-y\vert^{N}}+ \int_{\mathbb{R}^{N}}G(u)\ dx\right]\\
& \geq \left( 1-\frac{m}{\theta}\right)\min\lbrace\Vert u\Vert^L,\Vert u\Vert^m\rbrace.
\end{align*}
It follows, since $\theta>2m$, that $J(u) \rightarrow +\infty,$ when $\Vert u\Vert \rightarrow +\infty.$ \\

$(ii)$
 In light of assumptions $(f_1)-(f_2)$, we obtain that, for any $\varepsilon >0 $ there exists a positive constant $C_{\varepsilon}$ such that
\begin{equation}\label{3.1}
\vert f(t)t\vert \leq \varepsilon g(t)t +C_{\varepsilon}g_{*}(t)t,\ \text{for all} \  t \in \mathbb{R}.
\end{equation}
Using $(g_2)$ and Lemma \ref{lem2.5}, we get
\begin{equation}\label{3.2}
\vert f(t)t\vert \leq \varepsilon m G(t) +C_{\varepsilon}m^{*} G_{*}(t),\ \text{for all} \  t \in \mathbb{R}.
\end{equation}
Let $u\in\mathcal{N}$, so  $\langle J^{'}(u),u \rangle =0,$ that is,

\begin{align*}
& \int_{\mathbb{R}^{N}}\int_{\mathbb{R}^{N}}g\left( \frac{u(x)-u(y)}{\vert x-y\vert^{\alpha}}\right)\frac{u(x)-u(y)}{\vert x-y\vert^{\alpha +N}}dxdy+\int_{\mathbb{R}^{N}}g(u)udx\\
& =\int_{\mathbb{R}^{N}}K(x)f(u)udx.
\end{align*}

Exploiting $(K_1)$, $(g_1)$ and $(\ref{3.2})$, we get
\begin{align*}
& L\int_{\mathbb{R}^{N}}\int_{\mathbb{R}^{N}}G\left( \frac{u(x)-u(y)}{\vert x-y\vert^{\alpha}}\right)\frac{dxdy}{\vert x-y\vert^{N}}dxdy+L\int_{\mathbb{R}^{N}}G(u)dx\\
& \leq m\varepsilon\Vert K\Vert_{\infty}\int_{\mathbb{R}^{N}}G(u)dx+m^{*}C_{\varepsilon}\Vert K\Vert_{\infty}\int_{\mathbb{R}^{N}}G_{*}(u)dx,
\end{align*}
which is equivalent to
\begin{align*}
&[ L-m\varepsilon\Vert K\Vert_{\infty} ]\left[ \int_{\mathbb{R}^{N}}\int_{\mathbb{R}^{N}}G\left( \frac{u(x)-u(y)}{\vert x-y\vert^{\alpha}}\right)\frac{dxdy}{\vert x-y\vert^{N}}dxdy+ \int_{\mathbb{R}^{N}}G(u)dx \right]\\
& \leq m^{*}C_{\varepsilon}\Vert K\Vert_{\infty}\int_{\mathbb{R}^{N}}G_{*}(u)dx.
\end{align*}
Without lose of generality we may assume that
 $0\neq \Vert u\Vert <1$. Then by Lemma \ref{lem777} and Theorem \ref{thm2}, we deduce that
\begin{equation}\label{600}
[ L-m\varepsilon\Vert K\Vert_{\infty} ]\Vert u\Vert^{m}\leq m^{*}\tilde{C}C_{\varepsilon}\Vert K\Vert_{\infty}\Vert u\Vert^{m^{*}}.
\end{equation}
Hence, by choosing $\varepsilon$ small enough, we obtain
\begin{align*}
\left( \frac{C_{1}}{C_{2}}\right)^{\frac{1}{m^{*}-m}}& \leq \Vert u\Vert,
\end{align*}
where $C_{1}=L-m\varepsilon\Vert K\Vert_{\infty}>0$ and
$C_{2}=m^{*}\tilde{C}C_{\varepsilon}\Vert K\Vert_{\infty}>0.$
  Consequently, there exists a positive radius $\rho>0$ such that $\Vert u\Vert\geq\rho,$ with $\rho=\left( \frac{C_{1}}{C_{2}}\right)^{\frac{1}{m^{*}-m}}.$\\

Let $w\in\mathcal{M},$ $\langle J^{'}(w),w^{\pm}\rangle=0$. By Lemma \ref{lem400}, we see that $\langle J^{'}(w),w^{+}\rangle \geq \langle J^{'}(w^{+}),w^{+}\rangle,$ so
\begin{align*}
& \int_{\mathbb{R}^{N}}\int_{\mathbb{R}^{N}}g\left( \frac{w^{+}(x)-w^{+}(y)}{\vert x-y\vert^{\alpha}}\right)\frac{w^{+}(x)-w^{+}(y)}{\vert x-y\vert^{\alpha +N}}dxdy\\
 & +\int_{\mathbb{R}^{N}}g(w^{+})w^{+}\ dx\\
& \leq \int_{\mathbb{R}^{N}}K(x)f(w^{+})w^{+}\ dx,
 \end{align*}
 arguing as in the proof of the case $u\in \mathcal{N}$, we deduce the desired result.\\
\end{proof}


\begin{lemma}\label{lem3.2}
Let $\lbrace w_{n}\rbrace _{n} \subset \mathcal{M}$ such that $w_{n}\rightharpoonup w $ in $W^{\alpha,G}(\mathbb{R}^{N})$, then $w^{\pm}\neq 0.$
\end{lemma}
\begin{proof}[Proof]
By Lemma \ref{lem3.1}, there exists $\varrho>0$ such that
\begin{equation}\label{3.7}
\Vert w_{n}^{\pm}\Vert \geq\varrho,\ \ \text{for all}\ n\in\mathbb{N}.
\end{equation}
According to Lemma \ref{lem400} and have in mind that $w_{n}\in\mathcal{M}$, we obtain $$\langle J^{'}(w_{n}^{\pm}),w_{n}^{\pm} \rangle \leq\langle J^{'}(w_{n}),w_{n}^{\pm} \rangle =0.$$
 Applying Lemma \ref{lem777}, we infer that
\begin{equation}\label{3.6}
 L\min\lbrace \Vert w_{n}^{\pm}\Vert^{L},\Vert w_{n}^{\pm}\Vert^{m}\rbrace\leq \int_{\mathbb{R}^{N}}K(x)f(w_{n}^{\pm})w_{n}^{\pm}\ dx.
 \end{equation}
 Putting together $(\ref{3.7})$ and $(\ref{3.6}),$ we deduce that
 \begin{equation}\label{3.8}
  L\min\lbrace  \varrho^{L}, \varrho^{m}\rbrace\leq L\min\lbrace \Vert w_{n}^{\pm}\Vert^{L},\Vert w_{n}^{\pm}\Vert^{m}\rbrace\leq \int_{\mathbb{R}^{N}}K(x)f(w_{n}^{\pm})w_{n}^{\pm}\ dx.
 \end{equation}
 On the other hand, according to  Lemma \ref{lem222}-$(3)$, we have that
 \begin{equation}\label{3.9}
 \lim_{n\rightarrow+\infty}\int_{\mathbb{R}^{N}}K(x)f(w_{n}^{\pm})w_{n}^{\pm}\ dx= \int_{\mathbb{R}^{N}}K(x)f(w^{\pm})w^{\pm}\ dx.
 \end{equation}
Combining $(\ref{3.8})$ with $(\ref{3.9}),$ we get
 $$0< L\min\lbrace  \varrho^{L}, \varrho^{m}\rbrace\leq \int_{\mathbb{R}^{N}}K(x)f(w^{\pm})w^{\pm}\ dx,$$
 thus, we conclude that  $w^{\pm}\neq 0.$ Thus we prove the desired result.
\end{proof}
In the following, we are able to prove that every nontrivial sign-changing function in $W^{\alpha,G}(\mathbb{R}^{N})$ corresponds to a suitable function in $\mathcal{M}$.
\begin{lemma}\label{lem3.3}
Let $w\in W^{\alpha,G}(\mathbb{R}^{N}),\ w^{\pm}\neq 0,$ then there exist $s,t>0$ such that
$$\langle J'(tw^{+}+sw^{-}),w^{+}\rangle=0\ \text{and}\ \langle J'(tw^{+}+sw^{-}),w^{-}\rangle=0.$$
As a consequence $tw^{+}+sw^{-}\in \mathcal{M}.$
\end{lemma}
\begin{proof}[Proof]
Let $\varphi:\ (0,+\infty)\times  (0,+\infty) \rightarrow \mathbb{R}^{2}$ be a continuous vector field given by
$$\varphi(t,s)=\big{(}\varphi_{1}(t,s),\varphi_{2}(t,s)\big{)},$$
where $$\varphi_{1}(t,s)=\langle J'(tw^{+}+sw^{-}),tw^{+}\rangle\ \text{ and}\ \varphi_{2}(t,s)=\langle J'(tw^{+}+sw^{-}),sw^{-}\rangle, \ \ \text{for all}\ \ t,s \in (0,+\infty)\times (0,+\infty).$$
Let  $t,s>0$, we observe that
\begin{align*}
 \varphi_{1}(t,s) & =\langle J^{'}(tw^{+}+sw^{-}),tw^{+}\rangle \\
& = \int_{\mathbb{R}^{N}}\int_{\mathbb{R}^{N}}g\left( \frac{tw^{+}(x)-tw^{+}(y)}{\vert x-y\vert^{\alpha}}\right)\frac{tw^{+}(x)-tw^{+}(y)}{\vert x-y\vert^{\alpha +N}}dxdy\\
& +\int_{\mathbb{R}^{N}}g(tw^{+})tw^{+}dx-\int_{\mathbb{R}^{N}}K(x)f(tw^{+})tw^{+}dx\\
& + \int_{B}\int_{A}g\left( \frac{tw^{+}(x)-sw^{-}(y)}{\vert x-y\vert^{\alpha}}\right)\frac{tw^{+}(x)}{\vert x-y\vert^{\alpha +N}}dxdy\\
& - \int_{B}\int_{A}g\left( \frac{tw^{+}(x)}{\vert x-y\vert^{\alpha}}\right)\frac{tw^{+}(x)}{\vert x-y\vert^{\alpha +N}}dxdy\\
& +\int_{A}\int_{B}g\left( \frac{sw^{-}(x)-tw^{+}(y)}{\vert x-y\vert^{\alpha}}\right)\frac{-tw^{+}(y)}{\vert x-y\vert^{\alpha +N}}dxdy\\
& -\int_{A}\int_{B}g\left( \frac{-tw^{+}(y)}{\vert x-y\vert^{\alpha}}\right)\frac{-tw^{+}(y)}{\vert x-y\vert^{\alpha +N}}dxdy\\
\end{align*}
and
\begin{align*}
\varphi_{2}(t,s) & =\langle J^{'}(tw^{+}+sw^{-}),sw^{-}\rangle \\
& = \int_{\mathbb{R}^{N}}\int_{\mathbb{R}^{N}}g\left( \frac{sw^{-}(x)-sw^{-}(y)}{\vert x-y\vert^{\alpha}}\right)\frac{sw^{-}(x)-sw^{-}(y)}{\vert x-y\vert^{\alpha +N}}dxdy\\
& +\int_{\mathbb{R}^{N}}g(sw^{-})sw^{-}dx-\int_{\mathbb{R}^{N}}K(x)f(sw^{-})sw^{-}dx\\
& + \int_{B}\int_{A}g\left( \frac{tw^{+}(x)-sw^{-}(y)}{\vert x-y\vert^{\alpha}}\right)\frac{-sw^{-}(y)}{\vert x-y\vert^{\alpha +N}}dxdy\\
& - \int_{B}\int_{A}g\left( \frac{-sw^{-}(y)}{\vert x-y\vert^{\alpha}}\right)\frac{-sw^{-}(y)}{\vert x-y\vert^{\alpha +N}}dxdy\\
& +\int_{A}\int_{B}g\left( \frac{sw^{-}(x)-tw^{+}(y)}{\vert x-y\vert^{\alpha}}\right)\frac{sw^{-}(x)}{\vert x-y\vert^{\alpha +N}}dxdy\\
& -\int_{A}\int_{B}g\left( \frac{sw^{-}(x)}{\vert x-y\vert^{\alpha}}\right)\frac{sw^{-}(x)}{\vert x-y\vert^{\alpha +N}}dxdy.\\
\end{align*}
 By  $(\ref{3.2})$, Lemma \ref{lem400} and taking into account the assumptions $(K_1)$ and $(g_2)$, we infer that
\begin{align*}
  \varphi_{1}(t,t) & \geq \int_{\mathbb{R}^{N}}\int_{\mathbb{R}^{N}}g\left( \frac{tw^{+}(x)-tw^{+}(y)}{\vert x-y\vert^{\alpha}}\right)\frac{tw^{+}(x)-tw^{+}(y)}{\vert x-y\vert^{\alpha +N}}dxdy\\
& +\int_{\mathbb{R}^{N}}g(tw^{+})tw^{+}dx-\int_{\mathbb{R}^{N}}K(x)f(tw^{+})tw^{+}dx\\
& \geq L \int_{\mathbb{R}^{N}}\int_{\mathbb{R}^{N}}G\left( \frac{tw^{+}(x)-tw^{+}(y)}{\vert x-y\vert^{\alpha}}\right)\frac{dxdy}{\vert x-y\vert^{N}}\\
 &  + [L-\varepsilon m\Vert K\Vert_{\infty}]\int_{\mathbb{R}^{N}}G(tw^{+})\ dx\\
&  -m^{*}C_{\varepsilon}\Vert K\Vert_{\infty}\int_{\mathbb{R}^{N}}G_{*}(tw^{+})\ dx\\
\end{align*}
and
\begin{align*}
  \varphi_{2}(t,t) & \geq \int_{\mathbb{R}^{N}}\int_{\mathbb{R}^{N}}g\left( \frac{tw^{-}(x)-tw^{-}(y)}{\vert x-y\vert^{\alpha}}\right)\frac{tw^{-}(x)-tw^{-}(y)}{\vert x-y\vert^{\alpha +N}}dxdy\\
& +\int_{\mathbb{R}^{N}}g(tw^{-})tw^{-}dx-\int_{\mathbb{R}^{N}}K(x)f(tw^{-})tw^{-}dx\\
& \geq L \int_{\mathbb{R}^{N}}\int_{\mathbb{R}^{N}}G\left( \frac{tw^{-}(x)-tw^{-}(y)}{\vert x-y\vert^{\alpha}}\right)\frac{dxdy}{\vert x-y\vert^{N}}\\
&  + [L-\varepsilon m\Vert K\Vert_{\infty}]\int_{\mathbb{R}^{N}}G(tw^{-})\ dx\\
&  -m^{*}C_{\varepsilon}\Vert K\Vert_{\infty}\int_{\mathbb{R}^{N}}G_{*}(tw^{-})\ dx.\\
\end{align*}
Exploiting Lemma \ref{lem2.3} and Lemma \ref{lem2.6}, for $\varepsilon$ small enough , we find that
\begin{align*}
 \varphi_{1}(t,t) & \geq [  L-\varepsilon m\Vert K\Vert_{\infty}] \min\lbrace t^{L},t^{m}\rbrace\left[ \int_{\mathbb{R}^{N}}\int_{\mathbb{R}^{N}}G\left( \frac{w^{+}(x)-w^{+}(y)}{\vert x-y\vert^{\alpha}}\right)\frac{dxdy}{\vert x-y\vert^{N}}\right.\\
\ &  +\left.  \int_{\mathbb{R}^{N}}G(w^{+})\ dx\right]
 - m^{*}C_{\varepsilon}\Vert K\Vert_{\infty}\max\lbrace t^{L^{*}},t^{m^{*}}\rbrace\int_{\mathbb{R}^{N}}G_{*}(w^{+})\ dx\\
\end{align*}
and
\begin{align*}
 \varphi_{2}(t,t) & \geq [  L-\varepsilon m\Vert K\Vert_{\infty}] \min\lbrace t^{L},t^{m}\rbrace\left[ \int_{\mathbb{R}^{N}}\int_{\mathbb{R}^{N}}G\left( \frac{w^{-}(x)-w^{-}(y)}{\vert x-y\vert^{\alpha}}\right)\frac{dxdy}{\vert x-y\vert^{N}}\right.\\
\ & + \left.\int_{\mathbb{R}^{N}}G(w^{-})\ dx\right]
 - m^{*}C_{\varepsilon}\Vert K\Vert_{\infty}\max\lbrace t^{L^{*}},t^{m^{*}}\rbrace\int_{\mathbb{R}^{N}}G_{*}(w^{-})\ dx.\\
\end{align*}
Since $L<m<L^{*}<m^{*},$  there is $r_1>0$ small enough such that
\begin{equation}\label{mir1}
\varphi_{1}(t,t) >0,\ \varphi_{2}(t,t) >0\  \text{for all}\  t\in(0,r_1).
\end{equation}
By assumption $(f_3)$, there exists a positive constant $D>0$ such that
\begin{equation}\label{3.10}
F(t)\geq Dt^{\theta},\ \text{for every} \ t\ \text{sufficiently large.}
\end{equation}
Using Lemma \ref{lem2.4} and $(\ref{3.10})$ , we get
\begin{align*}
 \varphi_{1}(t,t) & \leq \max\lbrace t^{L},t^{m} \rbrace \left[\int_{\mathbb{R}^{N}}\int_{\mathbb{R}^{N}}g\left( \frac{w^{+}(x)-w^{+}(y)}
 {\vert x-y\vert^{\alpha}}\right)\frac{w^{+}(x)-w^{+}(y)}{\vert x-y\vert^{\alpha +N}}dxdy\right.\\
& + \int_{\mathbb{R}^{N}}g(w^{+})w^{+}dx + \int_{B}\int_{A}g\left( \frac{w^{+}(x)-w^{-}(y)}{\vert x-y\vert^{\alpha}}\right)\frac{w^{+}(x)}{\vert x-y\vert^{\alpha +N}}dxdy\\
& +\left.\int_{A}\int_{B}g\left( \frac{w^{-}(x)-w^{+}(y)}{\vert x-y\vert^{\alpha}}\right)\frac{-w^{+}(y)}{\vert x-y\vert^{\alpha +N}}dxdy\right]\\
& -\min\lbrace t^{L},t^{m} \rbrace \left[ \int_{B}\int_{A}g\left( \frac{w^{+}(x)}{\vert x-y\vert^{\alpha}}\right)\frac{w^{+}(x)}{\vert x-y\vert^{\alpha +N}}dxdy\right.\\
& - \left. \int_{A}\int_{B}g \left( \frac{-w^{+}(y)}{\vert x-y\vert^{\alpha}}\right)\frac{-w^{+}(y)}{\vert x-y\vert^{\alpha +N}}dxdy \right] \\
& -Dt^{\theta}\int_{\mathbb{R}^{N}}K(x)\vert w^{+}\vert ^{\theta} dx
\end{align*}
and
\begin{align*}
 \varphi_{2}(t,t) & \leq \max\lbrace t^{L},t^{m}\rbrace\left[\int_{\mathbb{R}^{N}}\int_{\mathbb{R}^{N}}g\left( \frac{w^{-}(x)-w^{-}(y)}{\vert x-y\vert^{\alpha}}\right)
 \frac{w^{-}(x)-w^{-}(y)}{\vert x-y\vert^{\alpha +N}}dxdy\right.\\
& + \int_{\mathbb{R}^{N}}g(w^{-})w^{-}dx +\int_{B}\int_{A}g\left( \frac{w^{+}(x)-w^{-}(y)}{\vert x-y\vert^{\alpha}}\right)\frac{-w^{-}(y)}{\vert x-y\vert^{\alpha +N}}dxdy\\
& +\left.\int_{A}\int_{B}g\left( \frac{w^{-}(x)-w^{+}(y)}{\vert x-y\vert^{\alpha}}\right)\frac{w^{-}(x)}{\vert x-y\vert^{\alpha +N}}dxdy\right]\\
& -\min\lbrace t^{L},t^{m}\rbrace\left[ \int_{B}\int_{A}g\left( \frac{-w^{-}(y)}{\vert x-y\vert^{\alpha}}\right)\frac{-w^{-}(y)}{\vert x-y\vert^{\alpha +N}}dxdy\right.\\
& -\left.
\int_{A}\int_{B}g\left( \frac{w^{-}(x)}{\vert x-y\vert^{\alpha}}\right)\frac{w^{-}(x)}{\vert x-y\vert^{\alpha +N}}dxdy\right] \\
& -Dt^{\theta}\int_{\mathbb{R}^{N}}K(x)\vert w^{+}\vert ^{\theta} dx.
\end{align*}
Since $\theta>m>L$,  there is $R_1>0$ sufficiently large such that
\begin{equation}\label{45}
\varphi_{1}(t,t)<0,\ \varphi_{2}(t,t)<0\ \text{for all}\  t\in(R_1,+\infty).
\end{equation}
In what follows, we prove that $\varphi_{1}(t, s)$ is increasing in $s$ on $(0,+\infty)$ for fixed $t>0$ and  $\varphi_{2}(t, s)$ is increasing in $t$ on $(0,+\infty)$ for fixed $s>0$.\\

 Let $s_{1},s_{2},t_{1},t_{2}\in (0,+\infty)$ such that $s_{1}\leq s_{2}$ and $t_{1}\leq t_{2}$. For fixed $t>0$ and $s>0$,  we have
\begin{align*}
& \varphi_{1}(t,s_{2})-\varphi_{1}(t,s_{1})\\
& =  \int_{B}\int_{A}g\left( \frac{tw^{+}(x)-s_{2}w^{-}(y)}{\vert x-y\vert^{\alpha}}\right)\frac{tw^{+}(x)}{\vert x-y\vert^{\alpha +N}}dxdy\\
& +\int_{A}\int_{B}g\left( \frac{s_{2}w^{-}(x)-tw^{+}(y)}{\vert x-y\vert^{\alpha}}\right)\frac{-tw^{+}(y)}{\vert x-y\vert^{\alpha +N}}dxdy\\
& - \int_{B}\int_{A}g\left( \frac{tw^{+}(x)-s_{1}w^{-}(y)}{\vert x-y\vert^{\alpha}}\right)\frac{tw^{+}(x)}{\vert x-y\vert^{\alpha +N}}dxdy\\
& -\int_{A}\int_{B}g\left( \frac{s_{1}w^{-}(x)-tw^{+}(y)}{\vert x-y\vert^{\alpha}}\right)\frac{-tw^{+}(y)}{\vert x-y\vert^{\alpha +N}}dxdy\\
& = \int_{B}\int_{A}\left[ g\left( \frac{tw^{+}(x)-s_{2}w^{-}(y)}{\vert x-y\vert^{\alpha}}\right)-g\left( \frac{tw^{+}(x)-s_{1}w^{-}(y)}{\vert x-y\vert^{\alpha}}\right)\right] \frac{tw^{+}(x)}{\vert x-y\vert^{\alpha +N}}dxdy\\
& +\int_{A}\int_{B}\left[ g\left( \frac{s_{2}w^{-}(x)-tw^{+}(y)}{\vert x-y\vert^{\alpha}}\right)-g\left( \frac{s_{1}w^{-}(x)-tw^{+}(y)}{\vert x-y\vert^{\alpha}}\right)\right] \frac{-tw^{+}(y)}{\vert x-y\vert^{\alpha +N}}dxdy\\
\end{align*}
and
\begin{align*}
& \varphi_{2}(t_{2},s)-\varphi_{2}(t_{1},s)\\
& =  \int_{B}\int_{A}g\left( \frac{t_{2}w^{+}(x)-sw^{-}(y)}{\vert x-y\vert^{\alpha}}\right)\frac{-sw^{-}(y)}{\vert x-y\vert^{\alpha +N}}dxdy\\
& +\int_{A}\int_{B}g\left( \frac{sw^{-}(x)-t_{2}w^{+}(y)}{\vert x-y\vert^{\alpha}}\right)\frac{sw^{-}(x)}{\vert x-y\vert^{\alpha +N}}dxdy\\
& - \int_{B}\int_{A}g\left( \frac{t_{1}w^{+}(x)-sw^{-}(y)}{\vert x-y\vert^{\alpha}}\right)\frac{-sw^{-}(y)}{\vert x-y\vert^{\alpha +N}}dxdy\\
& -\int_{A}\int_{B}g\left( \frac{sw^{-}(x)-t_{1}w^{+}(y)}{\vert x-y\vert^{\alpha}}\right)\frac{sw^{-}(x)}{\vert x-y\vert^{\alpha +N}}dxdy\\
& = \int_{B}\int_{A}\left[ g\left( \frac{t_{2}w^{+}(x)-sw^{-}(y)}{\vert x-y\vert^{\alpha}}\right)-g\left( \frac{t_{1}w^{+}(x)-sw^{-}(y)}{\vert x-y\vert^{\alpha}}\right)\right] \frac{-sw^{-}(y)}{\vert x-y\vert^{\alpha +N}}dxdy\\
& +\int_{A}\int_{B}\left[ g\left( \frac{sw^{-}(x)-t_{2}w^{+}(y)}{\vert x-y\vert^{\alpha}}\right)-g\left( \frac{sw^{-}(x)-t_{1}w^{+}(y)}{\vert x-y\vert^{\alpha}}\right)\right] \frac{sw^{-}(x)}{\vert x-y\vert^{\alpha +N}}dxdy.\\
\end{align*}
It follows, from assumption $(g_1)$, that
   $$\varphi_{1}(t,s_{2})-\varphi_{1}(t,s_{1})\geq 0\ \text{and}\ \varphi_{2}(t_{2},s)-\varphi_{2}(t_{1},s)\geq 0.$$
Then, $\varphi_{1}(t, s)$ is increasing in $s$ on $(0,+\infty)$ for fixed $t>0$ and  $\varphi_{2}(t, s)$ is increasing in $t$ on $(0,+\infty)$ for fixed $s>0$. By $(\ref{mir1})$ and $(\ref{45})$, there exist $r>0$ and $R>0$ with $r<R$ such that
\begin{align*}
\varphi_{1}(r,s)>0,\ & \varphi_{1}(R,s)<0\ \text{for all}\ s\in(r,R],\\
\varphi_{2}(t,r)>0,\ & \varphi_{2}(t,R)<0\ \text{for all}\ t\in(r,R].
\end{align*}

Applying Miranda theorem \cite{12}, there exist $t,s\in[r,R]$ such that $\varphi_1(t,s)=\varphi_2(t,s)=0$, which implies that $tw^{+}+sw^{-}\in \mathcal{M}.$
\end{proof}
Let $w\in W^{\alpha,G}(\mathbb{R})$ such that $w^{\pm}\neq 0$. We consider the functions\\ $h^{w}:\ [0,+\infty)\times [0,+\infty)\rightarrow \mathbb{R}$ and $\Phi^{w}: [0,+\infty) \times [0,+\infty)\rightarrow\mathbb{R}^{2}$, defined by
\begin{equation}\label{3.11}
h^{w}(t,s)=J(tw^{+}+sw^{-})
\end{equation}
and

\begin{align*}\label{3.12}
\Phi^{w}(t,s) & =\big{(}\Phi^{w}_{1}(t,s),\Phi^{w}_{2}(t,s)\big{)}\\
 & =\left( \frac{\partial h^{w}}{\partial t}(t,s),\frac{\partial h^{w}}{\partial s}(t,s)\right) \\
 & =\big{(} \langle J^{'}(tw^{+}+sw^{-}),w^{+}\rangle,\langle J^{'}(tw^{+}+sw^{-}),w^{-}\rangle\big{)}.
\end{align*}

The Jacobian matrix of $\Phi^{w}$ is
\begin{equation*}
(\Phi^{w})^{'}(t,s)=
\begin{pmatrix}
\frac{\partial \Phi^{w}_{1}}{\partial t}(t,s) &  \frac{\partial \Phi^{w}_{1}}{\partial s}(t,s)\\
\ & \ \\
\frac{\partial \Phi^{w}_{2}}{\partial t}(t,s) &  \frac{\partial \Phi^{w}_{2}}{\partial s}(t,s)
\end{pmatrix}.
\end{equation*}

In the following, we will prove that, if $w\in\mathcal{M},$ the function $h^{w}$ has a critical point in $[0,+\infty)\times [0,+\infty)$, precisely a global maximum in $(t, s) = (1, 1)$.
\begin{lemma}\label{lem3.4}
Let $w\in \mathcal{M},$ then
\begin{enumerate}
\item[(i)] $h^{w}(t,s)<h^{w}(1,1)=J(w),$ for all $t,s\geq 0$ such that $(t,s)\neq (1,1);$
\item[(ii)] $\det(\Phi^{w})^{'}(1,1)>0.$
\end{enumerate}
\end{lemma}
\begin{proof}[Proof]

 $(i)$  Let $w\in\mathcal{M},$  $\langle J^{'}(w),w^{\pm}\rangle=\langle J^{'}(w^{+}+w^{-}),w^{\pm}\rangle=0,$
thus
$$\Phi^{w}(1,1)=\big{(}\frac{\partial h^{w}(1,1)}{\partial t},\frac{\partial h^{w}(1,1)}{\partial s}\big{)}=(0,0),$$
so $(1,1)$ is a critical point of $h^{w}$.\\

Let prove that $h^{w}$ has a global maximum point in $[0,+\infty)\times [0,+\infty)$.
\begin{align*}
h^{w}(t,s) & = J(tw^{+}+sw^{-})\\
\ & =\int_{\mathbb{R}^{N}}\int_{\mathbb{R}^{N}}G\left( \frac{tw^{+}(x)+sw^{-}(x)-tw^{+}(y)-sw^{-}(y)}{\vert x-y\vert^{\alpha}}\right)\frac{dxdy}{\vert x-y\vert^{N}}\\
\ &  +  \int_{\mathbb{R}^{N}}G(tw^{+}+sw^{-})dx -\int_{\mathbb{R}^{N}}K(x)F(tw^{+}+sw^{-})dx\\
\ & = \tilde{\rho}(tw^{+}+sw^{-})-\int_{\mathbb{R}^{N}}K(x)F(tw^{+})dx-\int_{\mathbb{R}^{N}}K(x)F(sw^{-})dx\\
\end{align*}
Invoking Lemma \ref{lem777}, $(\ref{3.10})$ and $(K_1)$, we deduce,
for $t$ and $s$ large enough, that
\begin{align*}
h^{w}(t,s) & \leq \Vert tw^{+}+sw^{-}\Vert^{m}-D\Vert K\Vert_{\infty}\int_{\mathbb{R}^{N}}\vert tw^{+}\vert^{\theta}dx -D\Vert K\Vert_{\infty}\int_{\mathbb{R}^{N}}\vert sw^{-}\vert^{\theta}dx \\
& \leq 2^{m-1}\left( \vert t\vert^{m}\Vert w^{+}\Vert^{m}+\vert s\vert^{m}\Vert w^{-}\Vert^{m}\right)
-D\vert t\vert^{\theta}\Vert K\Vert_{\infty}\int_{\mathbb{R}^{N}}\vert w^{+}\vert^{\theta}dx\\
& -D\vert s\vert^{\theta}\Vert
K\Vert_{\infty}\int_{\mathbb{R}^{N}}\vert w^{-}\vert^{\theta}dx,
\end{align*}
which implies that $\lim\limits_{\vert (t,s)\vert\rightarrow +\infty}h^{w}(t,s)=-\infty.$ Here, we used the fact that $2m<\theta$.\\
By using the continuity of $h^{w}$, we  deduce the existence of a global maximum $(a,b) \in [0, +\infty) \times [0, +\infty)$ of $h^{w}$. \\

Let prove that $a, b>0.$ Suppose by contradiction that $b=0.$ Then
$\langle J^{'}(aw^{+}),aw^{+}\rangle=0$.
In light of Lemma \ref{lem2.4}, we have
\begin{align*}
\int_{\mathbb{R}^{N}}K(x)f(aw^{+})aw^{+}\ dx & \leq \max\lbrace a^{L},a^{m}\rbrace \left[\int_{\mathbb{R}^{N}}\int_{\mathbb{R}^{N}}g\left( \frac{w^{+}(x)-w^{+}(y)}{\vert x-y\vert^{\alpha}}\right)\frac{w^{+}(x)-w^{+}(y)}{\vert x-y\vert^{\alpha +N}}\right. \\
 & +\left.  \int_{\mathbb{R}^{N}}g(w^{+})w^{+}\ dx\right] \\
\end{align*}
which is equivalent to
\begin{align*}
\int_{\mathbb{R}^{N}}\frac{K(x)f(aw^{+})aw^{+}}{\max\lbrace a^{L},a^{m}\rbrace}\ dx & \leq \int_{\mathbb{R}^{N}}\int_{\mathbb{R}^{N}}g\left( \frac{w^{+}(x)-w^{+}(y)}{\vert x-y\vert^{\alpha}}\right)\frac{w^{+}(x)-w^{+}(y)}{\vert x-y\vert^{\alpha +N}}dxdy \\
&  + \int_{\mathbb{R}^{N}}g(w^{+})w^{+}\ dx.\\
\end{align*}
Hence, using Lemma \ref{lem2.3} and taking into account that $\displaystyle{\int_{\mathbb{R}^{N}}G(aw^{+})\ dx }>0$ (see Lemma \ref{lem3.1}-(ii) ), we infer that
\begin{align}\label{eq700}
\displaystyle{\int_{\mathbb{R}^{N}}\frac{K(x)f(aw^{+})aw^{+}}{\max\lbrace a^{L},a^{m}\rbrace^{2}}\ dx} & \leq \frac{\displaystyle{\int_{\mathbb{R}^{N}}G(w^{+})\ dx}}{\displaystyle{\int_{\mathbb{R}^{N}}G(aw^{+})\ dx}}\displaystyle{ \left[\int_{\mathbb{R}^{N}}\int_{\mathbb{R}^{N}}g\left( \frac{w^{+}(x)-w^{+}(y)}{\vert x-y\vert^{\alpha}}\right)\frac{w^{+}(x)-w^{+}(y)}{\vert x-y\vert^{\alpha +N}}dxdy\right.}\nonumber\\
& +\displaystyle{  \left. \int_{\mathbb{R}^{N}}g(w^{+})w^{+}\ dx \right].}\nonumber\\
\end{align}
Suppose that $a> 1,$ then
\begin{align}\label{eq800}
\displaystyle{\int_{\mathbb{R}^{N}}\frac{K(x)f(aw^{+})(w^{+})^{2m}}{(aw^{+})^{2m-1}}\ dx} & \leq \frac{\displaystyle{\int_{\mathbb{R}^{N}}G(w^{+})\ dx}}{\displaystyle{\int_{\mathbb{R}^{N}}G(aw^{+})\ dx}}\displaystyle{ \left[\int_{\mathbb{R}^{N}}\int_{\mathbb{R}^{N}}g\left( \frac{w^{+}(x)-w^{+}(y)}{\vert x-y\vert^{\alpha}}\right)\frac{w^{+}(x)-w^{+}(y)}{\vert x-y\vert^{\alpha +N}}dxdy\right.}\nonumber\\
& +\displaystyle{  \left. \int_{\mathbb{R}^{N}}g(w^{+})w^{+}\ dx \right].}\nonumber\\
\end{align}
On the other hand, since $\langle J^{'}(w),w^{+}\rangle=0,$ we have

\begin{align*}
\int_{\mathbb{R}^{N}}K(x)f(w^{+})w^{+}\ dx & = \int_{\mathbb{R}^{N}}\int_{\mathbb{R}^{N}}g\left( \frac{w(x)-w(y)}{\vert x-y\vert^{\alpha}}\right)\frac{w^{+}(x)-w^{+}(y)}{\vert x-y\vert^{\alpha +N}}dxdy\\
 & +\int_{\mathbb{R}^{N}}g(w)w^{+}\ dx.
\end{align*}
From Lemma \ref{lem100} (with $t=s=1$) and taking into account  that $\displaystyle{\int_{\mathbb{R}^{N}}G(w^{+})\ dx }>0,$ one has
\begin{align}\label{eq900}
 \displaystyle{\int_{\mathbb{R}^{N}}K(x)f(w^{+})w^{+}\ dx} & \geq \frac{\displaystyle{\int_{\mathbb{R}^{N}}G(w^{+})\ dx}}{\displaystyle{\int_{\mathbb{R}^{N}}G(w^{+})\ dx }}\displaystyle{ \left[ \int_{\mathbb{R}^{N}}\int_{\mathbb{R}^{N}}g\left( \frac{w^{+}(x)-w^{+}(y)}{\vert x-y\vert^{\alpha}}\right)\frac{w^{+}(x)-w^{+}(y)}{\vert x-y\vert^{\alpha +N}}dxdy \right. } \nonumber \\
 & + \displaystyle{\left. \int_{\mathbb{R}^{N}}g(w)w^{+}\ dx\right],}\nonumber\\
\end{align}
which is equivalent to
\begin{align}\label{eq999}
\displaystyle{\int_{\mathbb{R}^{N}}\frac{K(x)f(w^{+})(w^{+})^{2m}}{(w^{+})^{2m-1}}\ dx} & \geq \frac{\displaystyle{\int_{\mathbb{R}^{N}}G(w^{+})\ dx}}{\displaystyle{\int_{\mathbb{R}^{N}}G(w^{+})\ dx }}\displaystyle{ \left[ \int_{\mathbb{R}^{N}}\int_{\mathbb{R}^{N}}g\left( \frac{w^{+}(x)-w^{+}(y)}{\vert x-y\vert^{\alpha}}\right)\frac{w^{+}(x)-w^{+}(y)}{\vert x-y\vert^{\alpha +N}}dxdy \right. } \nonumber \\
 & + \displaystyle{\left. \int_{\mathbb{R}^{N}}g(w)w^{+}\ dx\right].}\nonumber\\
\end{align}
Putting together $(\ref{eq800})$ and $(\ref{eq999})$ and using assumption $(f_4)$ and Lemma \ref{lem2.8}, we find that
\begin{align}\label{h3}
 0<\frac{\delta}{\displaystyle{\int_{\mathbb{R}^{N}}G(w^{+})\ dx} }-\frac{\delta}{\displaystyle{\int_{\mathbb{R}^{N}}G(aw^{+})\ dx }}\leq \int_{\mathbb{R}^{N}}K(x)(w^{+})^{2m}\left[ \frac{f(w^{+})}{(w^{+})^{2m-1}}- \frac{f(aw^{+})}{(aw^{+})^{2m-1}}\right] \ dx<0
\end{align}
 where
 \begin{align*}
\delta & =\displaystyle{\int_{\mathbb{R}^{N}}G(w^{+})\ dx }\left[ \int_{\mathbb{R}^{N}}\int_{\mathbb{R}^{N}}g\left( \frac{w^{+}(x)-w^{+}(y)}{\vert x-y\vert^{\alpha}}\right)\frac{w^{+}(x)-w^{+}(y)}{\vert x-y\vert^{\alpha +N}}dxdy\right. \\
& +\left. \int_{\mathbb{R}^{N}}g(w^{+})w^{+}\ dx \right]> 0.
 \end{align*}
Thus the contradiction, so $a\leq 1.$

\ \\
\ \\
  Using Lemma \ref{lemgrand} and having in mind that $w\in\mathcal{M}$, we deduce that
\begin{equation}\label{eq55}
J(aw^{+})=h^{w}(a,0)< h^{w}(1,1).
\end{equation}

 Since $(a,b)$ is a global maximum point of $h^{w}$, a contradiction holds, so $b>0$. Similarly, we  show that $a>0.$\\

Let prove that $a\leq 1$ and $b\leq 1.$ Since $(h^{w})^{'}(a,b)=0,$ then
 $$\Phi_{1}^{w}(a,b)=\langle J^{'}(aw^{+}+bw^{-}),aw^{+}\rangle=0,$$
 that is,
 \begin{align*}
 & \int_{\mathbb{R}^{N}}\int_{\mathbb{R}^{N}}g\left( \frac{aw^{+}(x)+bw^{-}(x)-aw^{+}(y)-bw^{-}(y)}{\vert x-y\vert^{\alpha}}\right) \frac{aw^{+}(x)-aw^{+}(y)}{\vert x-y\vert^{\alpha +N}}dxdy\\
 & + \int_{\mathbb{R}^{N}}g(aw^{+})aw^{+}\ dx\\
 &= \int_{\mathbb{R}^{N}}K(x)f(aw^{+})aw^{+}\ dx.
 \end{align*}
 So
 \begin{align*}
 \int_{\mathbb{R}^{N}}K(x)f(aw^{+})aw^{+}\ dx & = \int_{A}\int_{A}g\left( \frac{aw^{+}(x)-aw^{+}(y)}{\vert x-y\vert^{\alpha}}\right) \frac{aw^{+}(x)-aw^{+}(y)}{\vert x-y\vert^{\alpha +N}}dxdy\\
 & + \int_{B}\int_{A}g\left( \frac{aw^{+}(x)-bw^{-}(y)}{\vert x-y\vert^{\alpha}}\right) \frac{aw^{+}(x)}{\vert x-y\vert^{\alpha +N}}dxdy\\
 & +\int_{A}\int_{B}g\left( \frac{bw^{-}(x)-aw^{+}(y)}{\vert x-y\vert^{\alpha}}\right) \frac{-aw^{+}(y)}{\vert x-y\vert^{\alpha +N}}dxdy\\
 & + \int_{\mathbb{R}^{N}}g(aw^{+})aw^{+}\ dx.
 \end{align*}
Without lose of generality, we can suppose that $a\geq b$. The
mapping $t\mapsto g(t)t$ is increasing and even, then
  \begin{align*}
 \int_{\mathbb{R}^{N}}K(x)f(aw^{+})aw^{+}\ dx & \leq \int_{A}\int_{A}g\left( \frac{aw^{+}(x)-aw^{+}(y)}{\vert x-y\vert^{\alpha}}\right) \frac{aw^{+}(x)-aw^{+}(y)}{\vert x-y\vert^{\alpha +N}}dxdy\\
 & + \int_{B}\int_{A}g\left( \frac{aw^{+}(x)-aw^{-}(y)}{\vert x-y\vert^{\alpha}}\right) \frac{aw^{+}(x)}{\vert x-y\vert^{\alpha +N}}dxdy\\
 & +\int_{A}\int_{B}g\left( \frac{aw^{-}(x)-aw^{+}(y)}{\vert x-y\vert^{\alpha}}\right) \frac{-aw^{+}(y)}{\vert x-y\vert^{\alpha +N}}dxdy\\
 &  +  \int_{\mathbb{R}^{N}}g(aw^{+})aw^{+}\ dx.
 \end{align*}
 By Lemma \ref{lem2.4}, we get
 \begin{align*}
& \int_{\mathbb{R}^{N}}K(x)f(aw^{+})aw^{+}\ dx \\
& \leq \max\lbrace a^{L},a^{m}\rbrace\left[ \int_{A}\int_{A}g\left( \frac{w^{+}(x)-w^{+}(y)}{\vert x-y\vert^{\alpha}}\right) \frac{w^{+}(x)-w^{+}(y)}{\vert x-y\vert^{\alpha +N}}dxdy\right.\\
 & + \int_{B}\int_{A}g\left( \frac{w^{+}(x)-w^{-}(y)}{\vert x-y\vert^{\alpha}}\right) \frac{w^{+}(x)}{\vert x-y\vert^{\alpha +N}}dxdy\\
 &  +   \int_{A}\int_{B}g\left( \frac{w^{-}(x)-w^{+}(y)}{\vert x-y\vert^{\alpha}}\right) \frac{-w^{+}(y)}{\vert x-y\vert^{\alpha +N}}dxdy \\
 & +\left.\int_{\mathbb{R}^{N}}g(w^{+})w^{+}\ dx\right] . \\
 \end{align*}
 So
  \begin{equation}\label{eq888}
  \int_{\mathbb{R}^{N}}K(x)f(aw^{+})aw^{+}\ dx \leq \max\lbrace a^{L},a^{m}\rbrace \sigma, \\
 \end{equation}
 where
 \begin{align*}
 \sigma & =\int_{A}\int_{A}g\left( \frac{w^{+}(x)-w^{+}(y)}{\vert x-y\vert^{\alpha}}\right) \frac{w^{+}(x)-w^{+}(y)}{\vert x-y\vert^{\alpha +N}}dxdy\\
 & + \int_{B}\int_{A}g\left( \frac{w^{+}(x)-w^{-}(y)}{\vert x-y\vert^{\alpha}}\right) \frac{w^{+}(x)}{\vert x-y\vert^{\alpha +N}}dxdy\\
 &  +\int_{A}\int_{B}g\left( \frac{w^{-}(x)-w^{+}(y)}{\vert x-y\vert^{\alpha}}\right) \frac{-w^{+}(y)}{\vert x-y\vert^{\alpha +N}}dxdy\\
 &  + \int_{\mathbb{R}^{N}}g(w^{+})w^{+}\ dx > 0.
 \end{align*}
 From $(\ref{eq888})$ and Lemma \ref{lem2.3}, we conclude that
 $$\int_{\mathbb{R}^{N}}K(x)\frac{f(aw^{+})aw^{+}}{\max\lbrace a^{L},a^{m}\rbrace^{2}}\leq \left[ \frac{1}{\int_{\mathbb{R}^{N}}G(aw^{+})\ dx}\int_{\mathbb{R}^{N}}G(w^{+})\ dx\right] \sigma.$$
 Suppose that $a>1,$ then
 \begin{equation}\label{100}
 \int_{\mathbb{R}^{N}}K(x)\frac{f(aw^{+})(w^{+})^{2m}}{ (aw^{+})^{2m-1}}\leq \left[ \frac{1}{\int_{\mathbb{R}^{N}}G(aw^{+})\ dx}\int_{\mathbb{R}^{N}}G(w^{+})\ dx\right] \sigma.
 \end{equation}
 On the other side, since $\langle J^{'}(w),w^{+}\rangle=0,$ we have
 \begin{align*}
  & \int_{\mathbb{R}^{N}}\int_{\mathbb{R}^{N}}g\left( \frac{w(x)-w(y)}{\vert x-y\vert^{\alpha}}\right) \frac{w^{+}(x)-w^{+}(y)}{\vert x-y\vert^{\alpha +N}}dxdy\\
 & + \int_{\mathbb{R}^{N}}g(w^{+})w^{+}\ dx\\
 &= \int_{\mathbb{R}^{N}}K(x)f(w^{+})w^{+}\ dx,
 \end{align*}
 which is equivalent to
 \begin{align*}
 \int_{\mathbb{R}^{N}}K(x)f(w^{+})w^{+}\ dx & = \int_{A}\int_{A}g\left( \frac{w^{+}(x)-w^{+}(y)}{\vert x-y\vert^{\alpha}}\right) \frac{w^{+}(x)-w^{+}(y)}{\vert x-y\vert^{\alpha +N}}dxdy\\
 & + \int_{B}\int_{A}g\left( \frac{w^{+}(x)-w^{-}(y)}{\vert x-y\vert^{\alpha}}\right) \frac{w^{+}(x)}{\vert x-y\vert^{\alpha +N}}dxdy\\
 &  +\int_{A}\int_{B}g\left( \frac{w^{-}(x)-w^{+}(y)}{\vert x-y\vert^{\alpha}}\right) \frac{-w^{+}(y)}{\vert x-y\vert^{\alpha +N}}dxdy\\
 &  +  \int_{\mathbb{R}^{N}}g(w^{+})w^{+}\ dx\\
 & =\sigma.
 \end{align*}
 We get
 \begin{equation}\label{101}
 \int_{\mathbb{R}^{N}}K(x)\frac{f(w^{+})(w^{+})^{2m}}{ (w^{+})^{2m-1}}= \left[ \frac{1}{\int_{\mathbb{R}^{N}}G(w^{+})\ dx}\int_{\mathbb{R}^{N}}G(w^{+})\ dx\right] \sigma.
 \end{equation}
 Putting together $(\ref{100})$ and $(\ref{101})$ and using $(f_4)$ and Lemma \ref{lem2.8}, we obtain

 \begin{align*}
0  & \leq \displaystyle{\left[ \frac{1}{\int_{\mathbb{R}^{N}}G(w^{+})\ dx}-\frac{1}{\int_{\mathbb{R}^{N}}G(aw^{+})\ dx}\right]\sigma\int_{\mathbb{R}^{N}}G(w^{+})\ dx} \\
& \leq\displaystyle{\int_{\mathbb{R}^{N}}K(x)(w^{+})^{2m}\left[ \frac{f(w^{+})}{ (w^{+})^{2m-1}}-\frac{f(aw^{+})}{ (aw^{+})^{2m-1}}\right] dx} \\
& < 0.
 \end{align*}

Thus the contradiction, then $0<b,a \leq 1.$

on the other side, in light of Lemma \ref{lemgrand}, we have
$$h^{w}(t,s)<h^{w}(1,1)\ \text{for every} \ (t,s)\in [0,1]\times [0,1]\backslash \lbrace (1,1)\rbrace.$$
Thus the proof of $(i)$.\\

 Proof of $(ii)$. By a simple computation, we get
\begin{align*}
 & \frac{\partial \Phi^{w}_{1}}{\partial t}(t,s)\\
 \ & = \int_{\mathbb{R}^{N}}\int_{\mathbb{R}^{N}}g^{'}\left( \frac{tw^{+}(x)+sw^{-}(x)-tw^{+}(y)-sw^{-}(y)}{\vert x-y\vert ^{\alpha}}\right)\frac{[w^{+}(x)-w^{+}(y)]^{2}}{\vert x-y\vert^{2\alpha +N}}dxdy\\
  \ & +\int_{\mathbb{R}^{N}}g^{'}(tw^{+})[w^{+}]^{2}\ dx-\int_{\mathbb{R}^{N}}K(x)f^{'}(tw^{+})[w^{+}]^{2}\ dx,
\end{align*}

\begin{align*}
& \frac{\partial \Phi^{w}_{2}}{\partial s}(t,s) \\
 & = \int_{\mathbb{R}^{N}}\int_{\mathbb{R}^{N}}g^{'}\left( \frac{tw^{+}(x)+sw^{-}(x)-tw^{+}(y)-sw^{-}(y)}{\vert x-y\vert ^{\alpha}}\right)\frac{[w^{-}(x)-w^{-}(y)]^{2}}{\vert x-y\vert^{2\alpha +N}}dxdy\\
& +\int_{\mathbb{R}^{N}}g^{'}(sw^{-})[w^{-}]^{2}\ dx-\int_{\mathbb{R}^{N}}K(x)f^{'}(sw^{-})[w^{-}]^{2}\ dx
\end{align*}
and
\begin{align*}
& \frac{\partial \Phi^{w}_{1}}{\partial s}(t,s)  = \frac{\partial \Phi^{w}_{2}}{\partial t}(t,s)\\
&  =\int_{\mathbb{R}^{N}}\int_{\mathbb{R}^{N}}g^{'}\left( \frac{tw^{+}(x)+sw^{-}(x)-tw^{+}(y)-sw^{-}(y)}{\vert x-y\vert ^{\alpha}}\right)\frac{[w^{+}(x)-w^{+}(y)][w^{-}(x)-w^{-}(y)]}{\vert x-y\vert^{2\alpha +N}}dxdy.\\
\end{align*}

For $t=s=1$, we have
\begin{align*}
 &\frac{\partial \Phi^{w}_{1}}{\partial t}(1,1) \\
 & = \int_{\mathbb{R}^{N}}\int_{\mathbb{R}^{N}}g^{'}\left( \frac{w(x)-w(y)}{\vert x-y\vert ^{\alpha}}\right)\frac{[w^{+}(x)-w^{+}(y)]^{2}}{\vert x-y\vert^{2\alpha +N}}dxdy\\
& + \int_{\mathbb{R}^{N}}g{'}(w^{+})[w^{+}]^{2}\ dx-\int_{\mathbb{R}^{N}}K(x)f^{'}(w^{+})[w^{+}]^{2}\ dx,
\end{align*}
\begin{align*}
& \frac{\partial \Phi^{w}_{2}}{\partial s}(1,1) \\
 & = \int_{\mathbb{R}^{N}}\int_{\mathbb{R}^{N}}g^{'}\left( \frac{w(x)-w(y)}{\vert x-y\vert ^{\alpha}}\right)\frac{[w^{-}(x)-w^{-}(y)]^{2}}{\vert x-y\vert^{2\alpha +N}}dxdy\\
& +\int_{\mathbb{R}^{N}}g{'}(w^{-})[w^{-}]^{2}\ dx-\int_{\mathbb{R}^{N}}K(x)f^{'}(w^{-})[w^{-}]^{2}\ dx,
\end{align*}
and
\begin{align*}
 &\frac{\partial \Phi^{w}_{1}}{\partial s}(1,1)  = \frac{\partial \Phi^{w}_{2}}{\partial t}(1,1)\\
 & =\int_{\mathbb{R}^{N}}\int_{\mathbb{R}^{N}}g^{'}\left( \frac{w(x)-w(y)}{\vert x-y\vert ^{\alpha}}\right)\frac{[w^{+}(x)-w^{+}(y)][w^{-}(x)-w^{-}(y)]}{\vert x-y\vert^{2\alpha +N}}dxdy.\\
\end{align*}

Denote by
\begin{align*}
 a_{1}= & \int_{\mathbb{R}^{N}}\int_{\mathbb{R}^{N}}g^{'}\left( \frac{w(x)-w(y)}{\vert x-y\vert ^{\alpha}}\right)\frac{[w^{+}(x)-w^{+}(y)]^{2}}{\vert x-y\vert^{2\alpha +N}}dxdy
 +\int_{\mathbb{R}^{N}}g{'}(w^{+})[w^{+}]^{2}\ dx,\\
 a_{2}= & \int_{\mathbb{R}^{N}}K(x)f^{'}(w^{+})[w^{+}]^{2}\ dx, \\
  a_{3}= & \int_{\mathbb{R}^{N}}K(x)f(w^{+})[w^{+}]\ dx, \\
  b_{1}= & \int_{\mathbb{R}^{N}}\int_{\mathbb{R}^{N}}g^{'}\left( \frac{w(x)-w(y)}{\vert x-y\vert ^{\alpha}}\right)\frac{[w^{-}(x)-w^{-}(y)]^{2}}{\vert x-y\vert^{2\alpha +N}}dxdy
 +\int_{\mathbb{R}^{N}}g{'}(w^{-})[w^{-}]^{2}\ dx,\\
 b_{2}= & \int_{\mathbb{R}^{N}}K(x)f^{'}(w^{-})[w^{-}]^{2}\ dx, \\
  b_{3}= & \int_{\mathbb{R}^{N}}K(x)f(w^{-})[w^{-}]\ dx \\
  \text{and} & \\
 c = &   \int_{\mathbb{R}^{N}}\int_{\mathbb{R}^{N}}g^{'}\left( \frac{w(x)-w(y)}{\vert x-y\vert ^{\alpha}}\right)\frac{[w^{+}(x)-w^{+}(y)][w^{-}(x)-w^{-}(y)]}{\vert x-y\vert^{2\alpha +N}}dxdy.
  \end{align*}

 By assumptions $(g_{2})$, $(f_{4})$ and $(K_{1})$, we  infer that

 \begin{equation}\label{eq3}
a_{1}>0,\ \  a_{2}>(2m-1)a_{3}>(m-1)a_{3}>0
\end{equation}
and
\begin{equation}\label{eq4}
 b_{1}>0,\ \  b_{2}>(2m-1)b_{3}>(m-1)b_{3}>0.
 \end{equation}
Since $w^{+}(x)w^{-}(y)+w^{+}(y)w^{-}(x)\leq 0,\ \text{for all}\ x,y\in\mathbb{R}^{N},$ then
$c\geq 0.$\\

 On the other side, we have
\begin{align*}
 & a_{1}+c=\int_{\mathbb{R}^{N}}\int_{\mathbb{R}^{N}}g^{'}\left( \frac{w(x)-w(y)}{\vert x-y\vert ^{\alpha}}\right)\frac{[w^{+}(x)-w^{+}(y)]^{2}}{\vert x-y\vert^{2\alpha +N}}dxdy
 +\int_{\mathbb{R}^{N}}g{'}(w^{+})[w^{+}]^{2}\ dx\\
 &  \ \ \ \ \ \ \ \ \  +\int_{\mathbb{R}^{N}}\int_{\mathbb{R}^{N}}g^{'}\left( \frac{w(x)-w(y)}{\vert x-y\vert ^{\alpha}}\right)\frac{[w^{+}(x)-w^{+}(y)][w^{-}(x)-w^{-}(y)]}{\vert x-y\vert^{2\alpha +N}}dxdy\\
  &  \ \ \ \ \ \ \ \ \ =\int_{\mathbb{R}^{N}}\int_{\mathbb{R}^{N}}g^{'}\left( \frac{w(x)-w(y)}{\vert x-y\vert ^{\alpha}}\right)\frac{[w^{+}(x)-w^{+}(y)][w(x)-w(y)]}{\vert x-y\vert^{2\alpha +N}}dxdy\\
  &  \ \ \ \ \ \ \ \ \ + \int_{\mathbb{R}^{N}}g{'}(w^{+})[w^{+}]^{2}\ dx
\end{align*}
Using the assumption $(g_{3})$ and having in mind  $\langle J^{'}(w),w^{\pm}\rangle=0$ $(w\in \mathcal{M}),$  we deduce that
\begin{align*}
a_{1}+c & \leq [m-1]\int_{\mathbb{R}^{N}}\int_{\mathbb{R}^{N}}g\left( \frac{w(x)-w(y)}{\vert x-y\vert ^{\alpha}}\right)\frac{[w^{+}(x)-w^{+}(y)]}{\vert x-y\vert^{\alpha +N}}dxdy\\
  & + [m-1]\int_{\mathbb{R}^{N}}g(w^{+})[w^{+}]\ dx\\
  & =[m-1]\int_{\mathbb{R}^{N}}K(x)f(w^{+})[w^{+}]\ dx,
\end{align*}

so
\begin{equation}\label{eq1}
a_{1}+c\leq [m-1]a_{3}.
\end{equation}
The same computation gives that
\begin{equation}\label{eq2}
b_{1}+c\leq [m-1]b_{3}.
\end{equation}
Putting together $(\ref{eq3}),(\ref{eq4}),(\ref{eq1})$ and  $(\ref{eq2}),$ we get
\begin{equation}\label{eq5}
[a_{1}-a_{2}]<[a_{1}-(m-1)a_{3}]\leq-c
\end{equation}
and
\begin{equation}\label{eq6}
[b_{1}-b_{2}]<[b_{1}-(m-1)b_{3}]\leq -c.
\end{equation}
From $(\ref{eq5})$ and  $(\ref{eq6}),$ we obtain
$$[a_{1}-a_{2}][b_{1}-b_{2}]>c^{2}.$$
Thus $$\det(\Phi^{w})^{'}(1,1)=[a_{1}-a_{2}][b_{1}-b_{2}]-c^{2}>0.$$
The proof of $(ii)$ is completed.
\end{proof}
Now we are able  to prove our main result.
\subsection{Proof of Theorem \ref{thm1}}

We will prove the existence of $w\in \mathcal{M}$ such that $J(w)=\inf_{v\in \mathcal{M}}J(v)=c_{0}$. By using a quantitative deformation lemma, we show that $w$ is a critical point
of $J,$ which is a sign-changing solution of $(P).$\\

In light of Proposition \ref{pro1} and Lemma \ref{lem3.1}, there is a  bounded minimizing sequence\\ $\lbrace w_{n}\rbrace_{n}\subset \mathcal{M}$, such that
\begin{equation}\label{4.1}
J(w_{n})\rightarrow \inf_{v\in \mathcal{M}}J(v)=:c_{0}>0
\end{equation}
and
$$w_{n}^{\pm}\rightharpoonup w^{\pm}\ \text{in}\ W^{\alpha,G}(\mathbb{R}^{N}),\
  w_{n}^{\pm}\rightarrow  w^{\pm}\ \text{a.e. in}\ \mathbb{R}^{N},\ w_{n}^{\pm}\rightarrow  w^{\pm}\ \text{in}\ L_{K}^{M}(\mathbb{R}^{N}).$$
By Lemma \ref{lem3.2}, we deduce that $w^{\pm}\neq 0,$ so $w=w^{+}+w^{-}$ is sign-changing function.\\ According to Lemma \ref{lem3.3}, there exist $s,t>0$ such that
\begin{equation}\label{4.2}
\langle J^{'}(tw^{+}+sw^{-}),w^{+}\rangle=0,\ \langle J^{'}(tw^{+}+sw^{-}),w^{-}\rangle=0,\ \ \text{and} \ tw^{+}+sw^{-} \in \mathcal{M}.
\end{equation}
Let prove that $s,t\leq 1$. Since $w_{n}\in\mathcal{M},$  $\langle J^{'}(w_{n}),w^{\pm}_{n}\rangle=0$, that is,
\begin{align}\label{4.3}
 & \int_{\mathbb{R}^{N}}\int_{\mathbb{R}^{N}}g\left( \frac{w_{n}(x)-w_{n}(y)}{\vert x-y\vert^{\alpha}}\right)\frac{w_{n}^{\pm}(x)-w_{n}^{\pm}(y)}{\vert x-y\vert^{\alpha +N}}dxdy\\
&+\int_{\mathbb{R}^{N}}g(w_{n}^{\pm})w_{n}^{\pm}dx \nonumber\\
 & =  \displaystyle{\int_{\mathbb{R}^{N}}K(x)f(w_{n}^{\pm})w_{n}^{\pm}dx}.\nonumber
\end{align}
By Fatou Lemma, it follows that
\begin{align}\label{4.5}
 & \int_{\mathbb{R}^{N}}\int_{\mathbb{R}^{N}}g\left( \frac{w(x)-w(y)}{\vert x-y\vert^{\alpha}}\right)\frac{w^{\pm}(x)-w^{\pm}(y)}{\vert x-y\vert^{\alpha +N}}dxdy\\
  & +\int_{\mathbb{R}^{N}}g(w^{\pm})w^{\pm}dx\nonumber \\
 & \leq \displaystyle{\liminf_{n\rightarrow +\infty}\int_{\mathbb{R}^{N}}K(x)f(w_{n}^{\pm})w_{n}^{\pm}dx}.\nonumber
\end{align}
In light of Lemma \ref{lem222}, we have
\begin{equation}\label{4.7}
\int_{\mathbb{R}^{N}}K(x)f(w_{n}^{\pm})w_{n}^{\pm}dx \rightarrow \int_{\mathbb{R}^{N}}K(x)f(w^{\pm})w^{\pm}dx.
\end{equation}
Combining $(\ref{4.5})$ and $(\ref{4.7})$, we deduce that
\begin{equation}\label{4.8}
\langle J'(w),w^{\pm}\rangle\leq 0.
\end{equation}
Putting together $(\ref{4.2})$ and $(\ref{4.8})$ and arguing as in the proof of Lemma \ref{lem3.4}-(a), we deduce that $s,t\leq 1$.\\

 Next we show that $J(tw^{+}+sw^{-})=c_{0}$ and $t=s=1.$\\
Exploiting the fact that $tw^{+}+sw^{-}\in \mathcal{M},\ w_{n}\in\mathcal{M},\ t,s\leq 1$ and by applying Lemma \ref{lemgrand}, Fatou Lemma and Lemma \ref{lem222}, we conclude that
\begin{align*}
c_{0}\leq J(tw^{+}+sw^{-})& = J(tw^{+}+sw^{-})-\frac{1}{2m}\langle J^{'}(tw^{+}+sw^{-}),tw^{+}+sw^{-}\rangle\\
&  \leq J(w)-\frac{1}{2m}\langle J^{'}(w),w\rangle\\
& \leq \lim\limits_{n\rightarrow +\infty}[J(w_{n})-\frac{1}{2m}\langle J^{'}(w_{n}),w_{n}\rangle]\\
& =\lim\limits_{n\rightarrow +\infty}J(w_{n})=c_{0}.
\end{align*}
Thus,  $tw^{+} + sw^{-} \in \mathcal{M}$ and $J(tw^{+}+sw^{-}) = c_0.$\\
Suppose that $0<t<1$ or $0<s<1$, then
\begin{align*}
c_{0}\leq J(tw^{+}+sw^{-})& = J(tw^{+}+sw^{-})-\frac{1}{2m}\langle J^{'}(tw^{+}+sw^{-}),tw^{+}+sw^{-}\rangle\\
&  < J(w)-\frac{1}{2m}\langle J^{'}(w),w\rangle\\
& \leq \lim\limits_{n\rightarrow +\infty}[J(w_{n})-\frac{1}{2m}\langle J^{'}(w_{n}),w_{n}\rangle]\\
& =\lim\limits_{n\rightarrow +\infty}J(w_{n})=c_{0},
\end{align*}
it yields a contradiction, thus $t=s=1$, so $w = w^{+} + w^{-} \in \mathcal{M}$ and $J(w^{+} + w^{-}) = c_{0}.$\\

It remains to show that $w$ is a critical point of $J,$ that is
$J^{'}(w) = 0.$  The proof is similar to the argument used in
\cite{5}. For the reader's convenience we will give the details. We
argue by contradiction, suppose that $w$ is a regular point of $J$,
thus $ J^{'}(w)\neq 0$. Then, there exist $\beta > 0$ and $v_{0} \in
W^{\alpha,G}(\mathbb{R}^{N})$ with $\Vert v_{0}\Vert = 1,$ such that
$\langle J^{'}(w),v_{0}\rangle =2\beta > 0$.  By the continuity of
$J^{'},$ we choose a
radius $R$ so that $\langle J^{'}(v),v_{0}\rangle=\beta>0$  for every $v\in B_{R}(w)\subset \mathbb{X}$ with $v^{\pm}\neq 0.$\\

Let us fix $D=(\frac{1}{2},\frac{3}{2})\times(\frac{1}{2},\frac{3}{2})\ \subset\mathbb{R}^{2}$, such that
\begin{enumerate}
\item[$(i)$] $(1,1)\in D$ and $\Phi^{w}(t,s)=(0,0)$ in $\bar{D}$ if and only if $(t,s)=(1,1);$
\item[$(ii)$] $c_{0}\notin h^{w}(\partial D);$
\item[$(iii)$]$\lbrace tw^{+}+sw^{-}:(t,s)\in \bar{D}\rbrace\subset B_{R}(w),$
\end{enumerate}
where $h^{w}$ and $\Phi^{w}$ are defined in Lemma \ref{lem3.4}. \\

Let $0<r<R$, such that
\begin{equation}\label{4.35}
\mathcal{B}=\overline{B_{r}(w)}\subset B_{R}(w)\ \text{and}\ \mathcal{B}\cap\lbrace tw^{+}+sw^{-}\ :(t,s)\in \partial D\rbrace=\emptyset.
\end{equation}

Now, let us define a continuous mapping $\rho:W^{\alpha,G}(\mathbb{R}^{N})\rightarrow [0,+\infty[$ such that $$\rho(u):=\text{dist}(u,\mathcal{B}^{c})\ \text{for all}\ u\in W^{\alpha,G}(\mathbb{R}^{N})$$
 and a bounded Lipschitz  vector field $V:W^{\alpha,G}(\mathbb{R}^{N})\rightarrow W^{\alpha,G}(\mathbb{R}^{N})$ given by $V(u)=-\rho(u)v_{0}$. According to the general theory of differential equations,
   the following Cauchy problem:
\begin{equation}\label{E}
\left \{
   \begin{array}{r c l}
      \eta^{'}(\tau,u)  & = & V(\eta(\tau,u)),\ \text{for all}\ \tau>0, \\
      \ & \ & \\
      \eta(0,u)   & = & u.
   \end{array}
   \right .
   \tag{E}
\end{equation}
admits a unique continuous solution  $\eta(.,u)$. Moreover, there exists $\tau_{0}>0$ such that for all $\tau\in [0,\tau_{0}]$ the following properties hold:
\begin{enumerate}
\item[$(a)$] $\eta(\tau,u)=u$ for all $u\notin\mathcal{B};$
\item[$(b)$]$\tau\rightarrow J(\eta(\tau,u))$ is decreasing for all $\eta(\tau,u)\in\mathcal{B};$
\item[$(c)$]$J(\eta(\tau,w))\leq J(w)-\frac{r\beta}{2}\tau_{0}.$
\end{enumerate}

Indeed, $(a)$ follows by the definition of $\rho.$ Regarding $(b),$ we  observe that\\ $\langle J^{'}(\eta(\tau,u)),v_{0}\rangle=\beta >0$ for $\eta(\tau,u)\in\mathcal{B}\subset B_{R}(w),$  by the definition of $\rho,$ we  infer that $\rho(\eta(\tau,u))>0$,  then
$$\frac{d}{d\tau}(J(\eta(\tau)))=\langle J^{'}(\eta(\tau)),\eta^{'}(\tau)\rangle=-\rho(\eta(\tau))\langle J^{'}(\eta(\tau)),v_{0}\rangle=-\rho(\eta(\tau))\beta<0,\ \text{for all}\ \eta(\tau)\in\mathcal{B},$$
that is $J(\eta(\tau,u))$ is decreasing with respect to $\tau.$\\
 $(c)$ Since $\eta(\tau,u)\in\mathcal{B}$ for every $\tau\in[0,\tau_{0}],$ we assume that
$$\Vert \eta(\tau,w)-w\Vert\leq \frac{r}{2}\ \text{for any}\ \tau\in[0,\tau_{0}].$$
It follow that $\rho(\eta(\tau,w))=\text{dist}(\eta(\tau,w),\mathcal{B}^{c})\geq \frac{r}{2},$ then
$$\frac{d}{d\tau}J(\eta(\tau,w))= -\rho(\eta(\tau,w))\beta\leq -\frac{r\beta}{2}.$$
Integrating on $[0,\tau_{0}],$ we get
$$J(\eta(\tau,w))-J(w)\leq -\frac{r\beta}{2}\tau_{0}.$$

We consider a suitable deformed path $\bar{\eta}_{\tau_{0}}:\overline{D}\rightarrow W^{\alpha,G}(\mathbb{R}^{N})$ defined by
$$\bar{\eta}_{\tau_{0}}(t,s):=\eta(\tau_{0},tw^{+}+sw^{-}),\ \text{for all}\ (t,s)\in \overline{D}.$$
We see that
$$\max_{(t,s)\in\overline{ D}}J(\bar{\eta}_{\tau_{0}}(t,s))<c_{0}.$$
Indeed, by $(b)$ and the fact that $\eta(0,u)=u,$ we get
\begin{align*}
J(\bar{\eta}_{\tau_{0}}(t,s)) & = J(\eta(\tau_{0},tw^{+}+sw^{-}))\leq J(\eta(0,tw^{+}+sw^{-}))\\
& =J(tw^{+}+sw^{-})=h^{w}(t,s)< c_{0}\ \text{for all}\ (t,s) \in \overline{D}\setminus\lbrace (1,1)\rbrace.
\end{align*}
For $(t,s)=(1,1),$ in virtue of $(c)$, we obtain
\begin{align*}
J(\bar{\eta}_{\tau_{0}}(1,1)) & = J(\eta(\tau_{0},w^{+}+w^{-}))= J(\eta(0,w))\\
& \leq J(w)-\frac{r\beta}{2}\tau_{0}<J(w)=h^{w}(1,1)= c_{0},
\end{align*}
thus, $\bar{\eta}_{\tau_{0}}(\overline{D})\cap \mathcal{M}=\emptyset,$ that is,
\begin{equation}\label{eq5.9}
\bar{\eta}_{\tau_{0}}(t,s)\notin \mathcal{M}\ \text{for all}\ (t,s)\in \overline{D}.
\end{equation}

On the other side, let $\Psi_{\tau_{0}}:\overline{D}\rightarrow \mathbb{R}^{2}$,
$$
\Psi_{\tau_{0}}:=\left( \frac{\langle J^{'}(\bar{\eta}_{\tau_{0}}(t,s)),(\bar{\eta}_{\tau_{0}}(t,s))^{+}\rangle}{t},\frac{\langle J^{'}(\bar{\eta}_{\tau_{0}}(t,s)),(\bar{\eta}_{\tau_{0}}(t,s))^{-}\rangle}{s}\right),
$$
 By $(\ref{4.35})$ and $(a)$, for all $(t,s)\in \partial D$ and $\tau=\tau_{0},$ we observe that
$$\Psi_{\tau_{0}}(t,s)=\left( \langle J^{'}(tw^{+}+sw^{-}),w^{+}\rangle,\langle J^{'}(tw^{+}+sw^{-}),w^{-}\rangle\right)=\Phi^{w}(t,s). $$
Using Brouwer's topological degree, we obtain
$$\text{deg}(\Psi_{\tau_{0}},D,(0,0))=\text{deg}(\Phi^{w},D,(0,0))=\text{sgn}(\det(\Phi^{w})^{'}(1,1))=1,$$
 $\Psi_{\tau_{0}}$ has a zero $(\bar{t},\bar{s})\in D$, namely
$$\Psi_{\tau_{0}}(\bar{t},\bar{s})=(0,0)\Longleftrightarrow\langle J^{'}(\bar{\eta}_{\tau_{0}}(\bar{t},\bar{s})),(\bar{\eta}_{\tau_{0}}(\bar{t},\bar{s}))^{\pm}\rangle=0.$$
Consequently there exists $(\bar{t},\bar{s})\in D$ such that $\bar{\eta}_{\tau_{0}}(\bar{t},\bar{s})\in \mathcal{M}$, thus the contradiction with $(\ref{eq5.9}).$
We conclude that $w$ is a critical point of $J.$\\




\bigskip
\noindent \textsc{\textsc{anouar bahrouni}} \\
Mathematics Department, Faculty of Sciences, University of Monastir,
5019 Monastir, Tunisia\\
 (Anouar.Bahrouni@fsm.rnu.tn; bahrounianouar@yahoo.fr)\\

\bigskip
\noindent \textsc{\textsc{hlel missaoui}} \\
Mathematics Department, Faculty of Sciences, University of Monastir,
5019 Monastir, Tunisia\\
 (hlelmissaoui55@gmail.com)

 \bigskip
\noindent \textsc{\textsc{hichem ounaies}} \\
Mathematics Department, Faculty of Sciences, University of Monastir,
5019 Monastir, Tunisia\\
 (hichem.ounaies@fsm.rnu.tn)

\end{document}